\documentclass[10pt,a4paper,reqno,oneside]{amsart}

\usepackage{amsrefs}

\usepackage{array}
\usepackage{url}
\usepackage{multicol}
\usepackage[linktocpage = true]{hyperref}
\usepackage[all,cmtip]{xy}

\hypersetup{
 colorlinks = true,
 citecolor = blue,
}
\usepackage{color}
\definecolor{red}{rgb}{1,0,0}
\definecolor{green}{rgb}{0,1,0}
\definecolor{blue}{rgb}{0,0,1}
\definecolor{refkey}{gray}{.625}
\definecolor{labelkey}{gray}{.625}
\usepackage[a4paper]{geometry}
\usepackage{amsmath,amssymb}
\usepackage{amssymb}
\usepackage{amsmath}
\usepackage{amsthm}
\usepackage{enumerate}
\usepackage{graphicx}

\usepackage{tikz-cd}
\allowdisplaybreaks
\usepackage[colorinlistoftodos,textsize=small]{todonotes} 

\makeatletter

\newtheorem{thmletter}{Theorem}

\theoremstyle{plain}
\newtheorem{thm}{\protect\theoremname}[section]

\newtheorem{cor}[thm]{\protect\corollaryname}
\newtheorem{lem}[thm]{\protect\lemmaname}

\theoremstyle{definition}

\newtheorem{defn}[thm]{\protect\definitionname}

\newtheorem{notation}[thm]{\protect\notationname}

\AtBeginDocument{
	
}

\makeatother

  \providecommand{\corollaryname}{Corollary}
  \providecommand{\examplename}{Example}
  \providecommand{\lemmaname}{Lemma}
  \providecommand{\propositionname}{Proposition}
  \providecommand{\theoremname}{Theorem}
  \providecommand{\definitionname}{Definition}
  \providecommand{\remarkname}{Remark}
  \providecommand{\notationname}{Notation}

  \makeatletter
  \newcommand{\be}{%
  \begingroup
  \eqnarray%
   \@ifstar{\nonumber}{}%
  }

\newcommand{\GL}{\mathrm{GL}}
\newcommand{\id}{\mathrm{Id}}
\newcommand{\dx}{\dot{\mathbf{x}}_0}
\newcommand{\x}{\mathbf{x}_0}

\newcommand{\A}{\mathcal{A}}
\newcommand{\e}{\mathbf{e}}

\newcommand{\F}{\mathcal{F}}
\newcommand{\G}{\mathcal{G}}
\newcommand{\Gal}{\operatorname{Gal}}
\newcommand{\Cl}{\operatorname{Cl}}
\newcommand{\n}{\mathfrak{n}}
\newcommand{\p}{\mathfrak{p}}

\newcommand{\floor}[1]{\lfloor #1 \rfloor}

\newcommand{\Prim}{\operatorname{Prim}}

\newcommand{\Ihara}{\operatorname{Ihara}}

\newcommand{\LT}{\operatorname{LT}}

\makeatother

\begin{document}

\title{New optimal function field towers over finite fields of quartic power}
%
\author{Chuangqiang Hu}
\address{Sun Yat-Sen University, School of Mathematical, Guangzhou, China}
\email{\href{huchq@mail2.sysu.edu.cn}{huchq@mail2.sysu.edu.cn}}
%

\author{Xiuwu Zhu}
\address{Beijing Institute of Mathematical Sciences and Applications, Beijing, China}
\email{\href{xwzhu@bimsa.cn}{xwzhu@bimsa.cn}}

\allowdisplaybreaks

\begin{abstract}
We introduce two new types of towers of Drinfeld modular curves. These towers originate from a specific domain $\A $ and are analogous to the towers of rank-two Drinfeld modular curves over the polynomial ring. Specifically, the domain $\A $ corresponds to the projective line over the finite field $ \mathbb{F}_q $, equipped with an infinite place of degree two. We select an arbitrary non-zero principal $\A $-ideal $ I_{\eta} $ of degree two. Notably, the $ I_{\eta} $-reduction of the tower of minimal Drinfeld modular curves is asymptotically optimal over the finite field $ \mathbb{F}_{q^4} $.
\end{abstract}
\maketitle{}


 
\section{Introduction}
\subsection{Ihara's quantity}
Let $\mathbb{F}_q$ denote the finite field of cardinality $q$. Let $F $ be a function field over $\mathbb{F}_q$ with genus $g(F)$.  
Estimating the number of rational places of $F $ is an important topic in number theory and algebraic geometry.
The Hasse-Weil bound \cite{MR29522} states that the number of rational places of $F$ satisfies the inequality 
\[
    N(F) \leqslant q+1 + 2\sqrt{q} g(F). 
\]  
An improved bound was obtained by Serre  \cite{Serre1984}:
 \[
N(F) \leqslant  q + 1 + g(F) \floor{2 \sqrt{q}}.
\] 
A function field $F / \mathbb{F}_q $ that achieves the Hasse-Weil bound is called maximal. For a comprehensive overview, the interested reader
is referred to \cites{Cakcak2004, Cakcak2005, MR2660419, MR2448446, MR2239917, MR2464941}
 for standard examples, and to \cites{MR3893192, Skabelund2018} for recent progress on maximal curves.

To investigate the asymptotic behavior for function fields over $\mathbb{F}_q$, 
Ihara \cite{Ihara1982} introduced the quantity 
\[
\Ihara(q) := \limsup_{g(F) \to \infty } \frac{N(F)}{g(F)},
\]
where $ F $ ranges over the function fields over $\mathbb{F}_q$.
 Due to Serre's lower bound \cite{Serre1983} and the Drinfeld-Vladut upper bound \cite{Drinfeld695100}, it is now well-known that 
  \[
0 < \Ihara(q)  \leqslant \sqrt{q}-1.
\]

In search of lower bounds for $\Ihara(q)$, researchers have invented various constructions of function field towers over $\mathbb{F}_q$. Roughly speaking,
a function field tower $\F_* $ means a sequence
of successive inclusions 
\[
    \begin{tikzcd}
       \F_0 \ar[r] &  \F_1\ar[r] &  \F_2\ar[r] &  \F_3\ar[r] &  \F_4\ar[r] &  \cdots,
   \end{tikzcd}
\]
of function fields over $ \mathbb{F}_q $ with $g(\F_n) \to \infty $ when $n \to \infty$. 

The limit
\[ \lambda(  \F_* ) = \lim_{n\to \infty} \frac{N(\F_n)}{g(\F_n)}  \]
 certainly gives a lower bound of $ \Ihara(q)$.
A function field tower $ \F_*$ is called asymptotically good if $\lambda(  \F_* ) > 0 $ and called asymptotically optimal if $\lambda(  \F_* )$ verifies the Drinfeld-Vladut bound. Notice that it is not generally easy to construct good towers.

When $q$ is a square, a sharp lower bound for Ihara's quantity is established independently: $\Ihara(q) \geqslant \sqrt{q}-1$  (hence $\Ihara(q) = \sqrt{q}-1$). Specifically, Ihara \cite{Ihara1982} used families of Shimura modular curves to derive this result, whereas Tsfasman, Vlăduț, and Zink \cite{MR705893} relied on families of classical modular curves. A key limitation of both approaches, however, is that the modular curves involved are not explicit. To address this lack of explicitness, Garcia and Stichtenoth \cite{Garcia1995} later constructed an explicit optimal sequence of function fields over \(\mathbb{F}_{q^2}\):
\[ \{ \F_n := \mathbb{F}_{q}(x_1,\cdots, x_n) \}  . \]
This tower is defined by a recursive condition on its variables $x_n$:
\[
    \frac{x_{n+1}}{x_n^q} +  \frac{x_{n+1}^q}{x_n} = 1.  
\]
In a subsequent work \cite{Garcia1996}, they extended this line of research by introducing a second explicit tower of function fields, governed by the recursive relation:
\[x_{n+1}^q + x_{n+1} = \frac{x_n^q}{x_n^{q-1} + 1}.\]
Later, Elkies \cite{MR1905359} established a key connection: the function field towers (and their associated curves) constructed by Garcia and Stichtenoth are in fact Drinfeld modular curves. This pattern extends to other settings:  well-performing function field constructions have also been shown to arise from modular curves. Motivated by this, Elkies conjectured that all optimal recursive towers must originate from some type of modular curve—though the exact formalization of this claim remains non-trivial to specify. Independently, Gekeler \cite{MR2037099} further confirmed the relevance of Drinfeld modular curves by showing that certain families of these curves also achieve the aforementioned lower bound (for Ihara's quantity).

When $q $ is not a square, the exact value of $\Ihara(q)$ remains undetermined.
Write $q = p^{2m+1}$ with $m \geqslant 1$. Bassa, Beelen, Garcia, and Stichtenoth \cite{Bassa2015} established a key lower bound: 
\[
\Ihara(q)\geqslant \frac{2(p^{m+1}-1)}{p+1+(p-1)/(p^m-1)} . \]
Notably, this bound is achieved by function field towers derived from Drinfeld modular curves of rank $2m+1$. Subsequently, the work in \cite{Nurdagul2017} and its follow-up \cite{MR4203564} built on this foundation. These contributions extended the analysis and provided a precise modular description for each function field in the towers.

\subsection{Main results}
In \cite{MR3287681} and \cite{MR3433893}, Bassa et al. systematically studied good families of Drinfeld modular curves. Roughly speaking, the points of the curve $X_0 (\n)$ parametrize isomorphism classes of pairs consisting of  rank-two Drinfeld $\A$-modules together with an $\n$-torsion, where the domain $\A$ is derived from a smooth algebraic curve over $\mathbb{F}_{q}$ associated with a marked place at infinity. 
The modular curve $ \x (\n)$ is defined as any geometric component of $X_0 (\n)$. Replacing the full set of Drinfeld modules with normalized modules in the sense of Hayes' notion \cite{MR535766},  
one can construct the normalized Drinfeld modular curves $\dx (\n)$ in a similar manner.
In particular, when $ \A = \mathbb{F}_{q}[t] $, and $\n = (t^n)$, the resulting Drinfeld modular curves $\x (\n)$ and $\dx (\n)$ coincide with the curves considered by Elkies \cite{MR1905359}  after reduction at $(t - 1) $.  

In this paper, we construct the optimal families of Drinfeld modular curves following the framework of Bassa et al.
We focus on the domain $\A$ corresponding to the projective line over the finite field $\mathbb{F}_q$ associated with an infinite place $P_{\rho}$ of degree two. One may show that the domain $ \A $ can be expressed as 
\[
    \A = \mathbb{F}_{q}[x,y ]/  \langle y^2 - ( \zeta+\zeta^q) xy + \zeta^{q+1} x^2 - x \rangle  
\]
where $\zeta$ is some element in $ \mathbb{F}_{q^2} \setminus \mathbb{F}_{q }$.
Recall that \cite{MR4808031} derived that the complete family of rank-two normalized Drinfeld modules is described by $ \phi^{\lambda}$ and its Frobenius twist $\phi^{\sigma;\lambda}$ with independent variable $\lambda  $.  Let $I_\infty = (x,y)$ be an ideal of $ \A$. 
The main result in this paper is devoted to computing the explicit expressions for the function field towers of modular curves with $ \n = I_\infty^n $.

\begin{thmletter}\label{thm:thmA}
    Let $ H^+ = \mathbb{F}_q(t, \zeta,   \nu )$ be the narrow class field of $ \A $ with the relation 
    \[
        \nu^{q+1} = -\frac{1}{(t - \zeta) (t^q - \zeta)}. 
    \]
    The normalized modular curves $ \dx(I_{\infty}^i ) $ are represented by the function field tower 
   \[
   \begin{tikzcd}
       \F_0 \ar[r, "q+1"] &  \F_1\ar[r, "q"] &  \F_2\ar[r, "q"] &  \F_3\ar[r, "q"] &  \F_4\ar[r, "q"] &  \cdots,
   \end{tikzcd}
\]
where the number on the arrow denotes the extension degree. 
The function field $ \F_k $ is generated by the variables $ \lambda_0, \cdots ,\lambda_k $ as follows.
\begin{enumerate}
    \item For $k = 0$, 
       $
        \F_0 = H^+ (\lambda_0 )$. 
    
    \item  For $k = 1$,
    $
        \F_1 = H^+ (\lambda_0 , \lambda_1 )
    $, where the defining equation of $ \lambda_0 $ and $ \lambda_1 $  is given by 
    \[\lambda_1^{q+1}-\lambda_0^{q}\lambda_1^{ q}-\frac{\zeta^{1-q}-1}{\zeta \lambda_0} (t -\zeta^q)\nu\lambda_1+\left( (\zeta^{-q}-\zeta^{-1} ) (t -\zeta^q)+(\zeta^{q-1}-1)^{q+1}\right)\nu\lambda_0^{q-1}=0.
\]
\item   For $k \geqslant 2 $, we obtain
\[
        \F_k =  \F_1(\lambda_2, \cdots, \lambda_k ) = H^+ (\lambda_0 , \lambda_1, \lambda_2 , \ldots, \lambda_k ),
    \]
    where the variables $ \lambda_2, \cdots , \lambda_k $ are subject to
    \[
    \sum_{i=0}^{q-1}\left(\frac{\nu^{\sigma^k} (1-\zeta^{1-q})^{q+1} \lambda_{k-2}^{ q-1}}{\lambda_{k-1}-\lambda_{k-2}^{ q}}\right)^i(\lambda_k-\lambda_{k-1}^{ q})^{q-i} -\nu^{\sigma^{k-1}-\sigma^{k-2}}\frac{\lambda_{k-1}-\lambda_{k-2}^{ q }}{\lambda_{k-2}^{q-1}}\lambda_{k-1}^{q-1} = 0 .
    \]
\end{enumerate}

    


\end{thmletter}

Similarly, we obtain the modular curves for minimal models.
\begin{thmletter}\label{thm:thmB}
  Let $ H $ be the class field of $ \A $. The minimal modular curves $\x(I_\infty^{k}) $ of Drinfeld $\A$-modules are represented by the function field tower 
   \[
   \begin{tikzcd}
       \G_0 \ar[r, "q+1"] &  \G_1\ar[r, "q"] &  \G_2\ar[r, "q"] &  \G_3\ar[r, "q"] &  \G_4\ar[r, "q"] &  \cdots.
   \end{tikzcd}
\]
The function field $ \G_k $ is generated by $ j_0 $ and $ w_i $ for $ i = 1 , \cdots , k$ in the following forms:
\begin{enumerate}
    \item For $k =0$, $ \G_0 = H(j_0)$.
    \item For $k =1$, $ \G_1 = H(j_0, w_1) = H(w_1)$, and the inclusion $ \G_0 \to \G_1 $ is represented by 
    \[
    j_0 \mapsto -\frac{1+\zeta^{-1} (t - \zeta^q)w_1 }{w_1^{ q+1}+(1-\zeta^{1-q})^{-1}w_1 }.
    \]
    
     \item For $k \geqslant 2$,
      $ \G_k = H(w_1, w_2, \ldots, w_k)$, where $w_i$ satisfy the relations 
        \[  \sum_{i=0}^{q-1}  (w_{k-1}^\nabla)^{i}w_2^{q-i}= \frac{w_{k-1}^q}{1-(\zeta^{q^{k+1}-q^k}-1)w_{k-1}  } \left( w_{k-1}^\nabla (t - \zeta^{q^{k+1}} ) \right)^{q-1} , \]
    and $w_{k-1}^{\nabla}$ is given by 
    \[
        w_{k-1}^{\nabla} = \frac{1}{(\zeta^{q^{k}-q^{k+1}}-1)(1+\zeta^{-q^k} (t - \zeta^{q^{k+1}})w_{k-1})}.
    \]
\end{enumerate}
Now the optimal family of function field towers over the finite field $ \mathbb{F}_{q^4} $ is derived from the reduction of the family in Theorem \ref{thm:thmB}.
\end{thmletter}
\begin{thmletter}\label{thm:thmC}
    Let $ I_{\eta} $ be a principal $\A$-ideal generated by $ z_\eta  $ with $ \deg z_\eta = 2 $. 
    Let $\x ( I_\infty^{k})/I_\eta $ ($k\geqslant 0$) denote the $I_\eta$-reduction of the minimal Drinfeld modular curves $\x ( I_\infty^{k})$. Then the genus of $ \x(I_\infty^{k})/I_\eta $ is given by 
    \begin{equation*} 
     g(\x(I_\infty^{k})/I_\eta) =   -1 + \frac{ q^{k-1} (q+1)  }{  q-1 } - \frac{ 2 }{q-1} \cdot ( q^{\floor{k/2}} + q ^{k - \floor{k/2} -1 }   -1  ) .
    \end{equation*}
    Consider $\x( I_\infty^{k})/I_\eta $ as defined over the constant field $\mathbb{F}_{\mathbf{q}} = \mathbb{F}_{q^4}$. Then all the supersingular points of these curves are $\mathbb{F}_{\mathbf{q}}$-rational. Moreover, the function field tower of $ \x (I_\infty^{k})/I_\eta $ is asymptotically optimal.
\end{thmletter}

\subsection{Remarks}
The approaches adopted in the paper are essentially the same as those in Elkies' paper \cite{MR1905359} (see \cites{Nurdagul2017, MR3287681,MR3433893,MR4203564} for further reading). 
We translate the information on primitive $I_\infty^k$-torsions into isogeny relations between the $k$-th Frobenius twists of Drinfeld modules for $ k\geqslant 0 $. 
However, we shall emphasize that the isogeny formula presents the main difficulty in this work.
Our main technique is to describe  explicitly the algebraic structure of $I_\infty$-annihilator $\phi_{I_\infty}^{\lambda}$, which is particularly studied in \cite{MR4808031} with the help of Anderson motives. 

It is remarkable that our Drinfeld modular tower also achieves Ihara's quantity.
From Goppa's construction \cite{MR628795}, good towers yield good linear error-correcting codes. By demonstrating that long linear codes can surpass the Gilbert-Varshamov bound \cite[Proposition 8.4.4]{MR2464941}, the celebrated work of Tsfasman et al. \cite{MR705893} established a crucial connection between coding theory and Ihara's quantity.
Recursive good towers play an important role in the study of Ihara's quantity, coding theory, and cryptography \cites{MR1849075, MR3225937, MR3724425, MR3493874, MR1426240,MR4762483}.

The paper is organized as follows. We introduce some necessary notations and results concerning Drinfeld modules in Section~\ref{sec:notations}. Section~\ref{sec:thmA} is devoted to constructing the tower of normalized Drinfeld module curves. Using this construction, we investigate the tower of minimal Drinfeld modules in Section~\ref{sec:thmB}. In Section~\ref{sec:thmC}, we compute the genus and rational places of the \( I_\eta \)-reduction of the tower to estimate Ihara's quantity.

\section{Preliminaries}\label{sec:notations}

\subsection{Galois groups and Hilbert fields}\label{sec:domain}
We briefly recall the necessary notations that will be used in the remainder of this paper.
Let $  \mathbb{P}^1  $ be the projective line over $ \mathbb{F}_q$. It is clear that the function field of $  \mathbb{P}^1  $ equals $ K = \mathbb{F}_{q}(t) $. 
Let $P_{\rho} $ be a degree two place of $ K $ corresponding to the monic irreducible polynomial $ \rho(t) = (t - \zeta)(t - \zeta^q) \in  \mathbb{F}_{q}[t]$ for some element $ \zeta \in \mathbb{F}_{q^2} \setminus \mathbb{F}_q $.
We view $P_{\rho} $ as the infinity of $ \mathbb{P}^1 $ for the remainder of this paper. Denote by $ \A $ the Dedekind domain arising from $ \mathbb{P}^1  - P_{\rho} $; that is  
\[
    \A = H^0(\mathbb{P}^1_{\mathbb{F}_q} - P_{\rho}, \mathcal{O}_{\mathbb{P}^1}).
\]
Notice that the quotient field of $\A$ recovers the function field $ K $.
Define the coordinates $x, y $ of $ \A $ as 
\[
     x = \frac{1}{\rho(t)} , \quad y = \frac{t}{\rho(t)}. 
\]
Then we know $\A$ is indeed the $\mathbb{F}_q$-algebra generated by $x,y$, that is  
\[
    \A = \mathbb{F}_{q}[x,y ]/  \langle y^2 - ( \zeta+\zeta^q) xy + \zeta^{q+1} x^2 - x \rangle. 
\]
From the class field theory for function fields, the Hilbert field of $ \A $ is given by 
\[ H = \mathbb{F}_q(t, \zeta) = \mathbb{F}_{q^2}(T)  ,\]
where $ T = \frac{1}{t -\zeta^q } = y - \zeta x $.
The Galois group of $ H/ K $ is isomorphic to $ \Cl(\A) = \mathbb{Z} /2  $ generated by 
\[
    \sigma : \zeta \mapsto \zeta^q. 
\]
In particular, $\sigma$ acts on $ T $ by
\[
    T^{\sigma} :=   \frac{1}{t - \sigma(\zeta^q)} = \frac{1}{t - \zeta} = \frac{T}{ 1 + (\zeta^q - \zeta)T}.
\]
Note that we always denote the action of $ \sigma $ by a superscript to simplify the notation.
Let $ H^+ = \mathbb{F}_q (t, \zeta , \nu) = H(\nu) $,
where $\nu$ satisfies 
\[
    \nu^{q+1} = - T^{q + \sigma}=  \frac{-1}{ (t - \zeta)(t^q - \zeta) }. 
\]
As shown in \cite{MR4808031}, $ H^+ $ is the narrow class field of $ K $ with respect to a sign function at $ P_{\rho}$.  
The Galois group $ \Gal(H^+/ K )$ is isomorphic to 
\[ \mathbb{Z}/ 2 \mathbb{Z} \times \mathbb{Z}/ (q+1)\mathbb{Z} \]
with generators $ \sigma $ and $ M_\mu $, where 
\begin{equation}\label{eq:nusigma}
    \sigma : \zeta \mapsto \zeta^q ; \quad \nu \mapsto \nu^{\sigma} := T^{1-q} \nu^q = - \frac{x}{\nu}
\end{equation}
and for $\mu^{q+1} = 1 $, 
\[
    M_{\mu} : \zeta \mapsto \zeta ; \quad  \nu \mapsto \mu \nu .
\]
We know the isomorphism $ \Cl^+(\A) \cong \Gal(H^+/ K ) $ is given by the Artin map of $ H^+/ K $. 
One can show that $ \sigma $ is in fact the image of $ I_\infty $. 
In particular, since $ I_\infty^2 = (x) $, we have $ \sigma^2 = \id $. 
\subsection{Drinfeld $\A$-modules}
For a Dedekind domain $ \A $ over $\mathbb{F}_q$, we define the norm $|a|$ of an element $ a \in \A $ to be 
the cardinality of $ \A /\langle a \rangle $. We define the degree  of $a \in \A$ to be $ \deg(a) := \log_{q}(|a|)$. 
Let $L $ be an $\A$-field, i.e., a field over $\mathbb{F}_q$ equipped with an $\mathbb{F}_q$-homomorphism $\iota: \A \to L$. The kernel of $\iota$ is called the $\A$-characteristic of $L$. Let $L\{ \tau \} $ be the non-commutative polynomial ring generated by $q$-Frobenius endomorphism $\tau $ satisfying 
\[\tau a = a^q \tau ,\]
for all $a  \in L $. 
The non-commutative polynomial ring $L\{ \tau \} $ is usually called the twisted polynomial ring over $L$.
\begin{defn}
We recall that the Drinfeld $ \A $-module of rank $r (\geqslant 1 )$ is defined as the morphism 
\[
    \phi : \A \to L\{ \tau \}
\]
of  $ \mathbb{F}_q$-algebras such that the image $ \phi_a $ of $ a \in \A $ is a twisted polynomial of $\tau$-degree $r \deg(a)$.
\end{defn}
In particular, when $\A$ is the Dedekind domain given in Section \ref{sec:domain}, we have $\deg(x) = \deg(y) = 2$. By definition, a rank-two Drinfeld $ \A$-module requires 
\[\deg_{\tau} \phi_x = \deg_{\tau} \phi_y= 4. 
\] 
Let $ \LT_{\phi}(a) $ be the leading coefficient of $ \phi_a $. We call $ \phi $  a normalized Drinfeld module (or a Hayes-Drinfeld module) if $ \LT_{\phi}(x)=1 $. 
It follows directly that $ \LT_{\phi}(y)/ \LT_{\phi}(x) $ equals either $\zeta $ or $\zeta^q $.
We call $ \phi $ a $\zeta$-type (resp. $\zeta^q$-type) Drinfeld $\A$-module if $ \LT_{\phi}(y)/ \LT_{\phi}(x) $ equals $ \zeta $ (resp. $\zeta^q$).
\begin{defn}
    Let $ \phi $ and $ \psi $ be Drinfeld $\A$-modules. 
    \begin{enumerate}
    \item We say $ \lambda \in L\{ \tau \}$  is an isogeny from $ \phi $ to $ \psi $ (over $ L $) if the equality 
    \[
        \lambda \phi_a = \psi_a \lambda 
    \]
    holds for all $a \in \A$. 
    \item In particular, if $ \lambda \in \bar{L}^* $, we say $ \phi $ is isomorphic to $ \psi $.
    \end{enumerate}
\end{defn}
Obviously, any Drinfeld $\A$-module is isomorphic to some normalized Drinfeld $\A$-module. 

\subsection{Normalized model}
The following theorem is one of the core results in \cite{MR4808031}, where detailed discussions can be found.
\begin{thm} \label{thm:normalized}
Let $\phi^\lambda$ be the $\zeta^q$-type degree two normalized Drinfeld module parameterized by $\lambda$, defined over $H^+(\lambda)$. Its expressions are as follows:
	\[
	\begin{cases}
	\phi_x := (\tau^2 + \tilde{\alpha} \tau + \frac{x}{\nu \lambda^{q-1}} ) (\tau^2+ \alpha \tau+ \nu \lambda^{q-1}),  \\
	\phi_{y} := \zeta^q (\tau^2 + \tilde{\beta} \tau + \frac{y}{\zeta^q \nu \lambda^{q-1}}  ) (\tau^2 + \alpha \tau+ \nu \lambda^{q-1}) ,
	\end{cases}	
	\]
where the coefficients satisfy:
	\begin{align*}
		\tilde{\alpha} & = -\frac{ \nu T^{\sigma q-(q+\sigma)}}{\zeta \lambda^{q^2}}	+ \frac{\zeta^{q}\lambda }{\zeta-\zeta^{q}}, 	\\
		\tilde{\beta}  & =-\frac{ \nu T^{\sigma q-(q+\sigma)}}{\zeta\lambda^{q^2}}+ \frac{\zeta\lambda }{\zeta-\zeta^{q}} , \\
\alpha  & =\frac{\lambda^{q^2}}{1-\zeta^{1-q}}+\frac{\nu }{\zeta T\lambda }.
\end{align*}
This module $\phi^\lambda$ is complete, meaning that any $\zeta^q$-type rank-two normalized Drinfeld module over $\bar{L}$ is isomorphic to $ \phi^{\lambda_0} $ for some $\lambda_0 \in \bar{L}$.
\end{thm}
The isomorphism class of $ \phi^\lambda $ is determined by the  $j$-invariant, namely
\[
    j (\phi^\lambda)= \frac{\lambda^{q^2+1}}{\nu}. 
\]
More precisely, the two Drinfeld modules of $ \zeta^q $-type $ \phi^{\lambda_1} $ and $ \phi^{\lambda_2} $ are isomorphic if and only if
\[
    j(\phi^{\lambda_1}) = j(\phi^{\lambda_2}). 
\]

Applying $ \sigma $ to the coefficients of $ \phi^{\lambda} $ yields the 
$ \zeta $-type Drinfeld module $  \phi ^{\sigma; \lambda' }   $ with the formal variable $ \lambda' := \lambda^\sigma $. 
This is in fact the complete family of $ \zeta $-type normalized modules of rank two. 
 Moreover, applying $ \sigma^k $, we have 
\[
    \phi^{\sigma^k; \lambda'} = \begin{cases}
    \phi^{\sigma; \lambda'}, & \text{ when $k$ is odd; }\\
    \phi^{  \lambda'}, & \text{ when $k$ is even. }\\
    \end{cases}
\]
Accordingly, the $j $-invariant of $\phi ^{\sigma^k; \lambda' } $ is defined as 
\[
    j(\phi ^{\sigma^k; \lambda' }) = \frac{\lambda'^{q^2+1}}{\nu^{\sigma^k}}.
\]

\subsection{Minimal model}\label{sec:minial}
In fact, the isomorphism class of rank-two Drinfeld $\A$-modules can be represented by the minimal model $\Phi^j$, which is parameterized by the $j$-invariant and defined over the Hilbert class field of $\A$.
If we choose $ \ell $ to be a root of $ \ell^{q-1} = \lambda^q $, then $ \ell $ is essentially an isogeny from $ \Phi^j $ to $ \phi^{\lambda} $. 
\begin{thm}\label{thm:Drinfeld}
   With the notations above, the complete family of $\zeta^q$-type Drinfeld $\A$-modules of rank two is represented as the minimal model $ \Phi^j $ (parameterized by $j$) in the following form:
\begin{align*}
\Phi_x^j =& \ell^{-1}\phi^\lambda_x\ell \\
=&\left(-j^{q(q+1)} T^{(\sigma+q)q}\tau^2+\left(\frac{T^{\sigma q}j^q}{\zeta}+\frac{j^{q+1}T^{\sigma+q}}{1-\zeta^{1-q}}\right)\tau +xj\right)\\
&\cdot\left(\tau^2+\left(\frac{1}{1-\zeta^{1-q}}+\frac{1}{\zeta T}\frac{1}{j}\right)\tau +\frac{1}{j}\right),
\end{align*}
and
\begin{align*}
\Phi_y^j =& \ell^{-1}\phi^\lambda_y\ell\\
 =  & \left(- j^{q(q+1)}T^{(\sigma+q)q}\zeta^q\tau^2+\left(\zeta^{q-1}j^q T^{\sigma q}-\frac{\zeta^q T^{\sigma+q}j^{q+1}}{1-\zeta^{q-1}}\right)\tau + yj\right)\\
 &\cdot\left(\tau^2+\left(\frac{1}{1-\zeta^{1-q}}+\frac{1}{\zeta T}\frac{1}{j}\right)\tau +\frac{1}{j}\right).
\end{align*}
So any $\zeta^q$-type rank-two Drinfeld $\A$-module is isomorphic to some $ \Phi^{j_0} $.
\end{thm}
Obviously, the smallest field of definition of $ \Phi^j $ is $ H(j) = \mathbb{F}_{q}(T, \zeta, j )$.
One may also obtain the minimal model of $\zeta$-type rank-two Drinfeld module $   \Phi ^{\sigma; j'}$ parameterized by the formal variable $ j' $ by applying the $\sigma $-action on the coefficients of $\Phi^j$. In the same manner, we adopt the symbol $  \Phi ^{\sigma^k; j'}$ for the $k$-th $\sigma$-action of $ \Phi ^{j}$. 
\subsection{Torsions}
Let $\phi$ be a rank $r$ Drinfeld $\A$-module over the $\A$-field $L$.
Let $I$ be an ideal of $ \A $. Denote by $\phi[I]$
the $I$-torsion of $ \phi $, i.e., the intersection of all kernels $ \ker \phi_{f} $ for all $f \in I $. Through the Drinfeld module $\phi$, the $I$-torsion admits a natural $\A$-module structure. If the $\A$-characteristic of $L$ is disjoint with $I$, then the $\A$-module structure on $\phi[I]$ is isomorphic to 
$
    (\A /I)^{\oplus r}
$.

\begin{defn}
   A primitive $ \n $-torsion of $ \phi $ is an $\A$-submodule of $ \phi[\n] $ which is isomorphic to $ \A / \n $. We denote by $   \Prim_{\n}(\phi) $ the set  collecting all the primitive $ \n $-torsion of $ \phi $.
\end{defn}

\begin{defn}
The maximal common monic right-divisor $\phi_I \in \bar{L}\{\tau\}$ of each $\phi_{f} $ for $f \in I $ is called the annihilator of the ideal $I$. 
\end{defn} 
It is trivial to check that the kernel of $\phi_I$ coincides with $\phi[I]$. 
Let $ I_\infty = (x, y ) $ and $ I_0 = ( x - \zeta^{-q-1}, y )$. 
Note that the twisted polynomial $ \tau^2+ \alpha \tau+ \nu \lambda^{q-1} $ is the common right-divisor of both $\phi_x^\lambda$ and $ \phi_y^\lambda $, we conclude that 
\begin{equation}\label{eq:phiIinf}
    \phi_{I_\infty}^{\lambda} = \tau^2+ \alpha \tau+ \nu \lambda^{q-1} .
\end{equation}
\begin{lem}[Equation (48) of \cite{MR4808031}]\label{lem:phi_I_infty_rep}
Let $\phi^{\lambda}$ be the $\zeta^q$-type normalized Drinfeld module introduced above. Then the annihilator $ \phi_{I_\infty}^{\lambda}$ verifies the equality:
\[\phi_{I_\infty}^{\lambda}=\frac{\tau+A(\lambda)}{C(\lambda)}\phi_y^{\lambda}-\frac{\zeta\tau+B(\lambda)}{C(\lambda)}\phi_x^{\lambda}\]
where
\[A(\lambda)=\frac{\zeta}{\zeta-\zeta^q}\lambda^q,\quad B(\lambda)=\frac{\zeta^2}{\zeta-\zeta^q}\lambda^q,\quad C(\lambda)=\frac{T}{1-\zeta^{q-1}}\frac{\lambda^q}{\delta}.\]
\end{lem}

Moreover, the annihilator  $ \phi_{I_0}^{\lambda}$ for the ideal $I_0$ is also determined in  \cite{MR4808031}: 
\[
\phi_{I_0}^{\lambda} = \phi^{\lambda}_{I_\infty} + \frac{(1-\zeta^{q-1}) \nu  }{\zeta^q T \lambda } \tau +  \frac{\nu \lambda^{q-1}}{\zeta^q T}.
\]
We remark that the explicit formula of $\phi^{\lambda}$ is indeed deduced from the expressions of $ \phi_{I_\infty}^{\lambda} $ and $ \phi_{I_0}^{\lambda} $ through technical computations.
For the minimal model $\Phi^j$, one may show that 
\[
\Phi_{I_\infty}^j =  \tau^2+\left(\frac{1}{1-\zeta^{1-q}}+\frac{1}{\zeta T}\frac{1}{j}\right)\tau +\frac{1}{j} 
\]
and 
\[
 \Phi_{I_0}^j   =   \tau^2+\left(\frac{1}{1-\zeta^{1-q}} + \frac{1   }{\zeta^q T j  } \right)\tau  +  \left(\frac{1 }{\zeta^q T   }+1 \right)\frac{1}{j }.  
\] 
\subsection{Modular curves}
Let $\A$ temporarily be a Dedekind domain with quotient field $ K $.
For  a congruence subgroup $ \Gamma  $ of $ \GL_2 (\A)$, the quotient of the Drinfeld upper plane by $ \Gamma $  classifies isomorphism classes of Drinfeld modules equipped with some "level structure". These quotients are the analogs of various modular curves classifying elliptic curves. 

For a nonzero ideal $ \n \in \A$, 
  Gekeler investigates in \cite{MR594434} (among other things) the Drinfeld modular curve $Y_0(\n)$, defined as the Drinfeld upper space modulo the arithmetic subgroup 
  \[
   \Gamma_0(\n) := \left\{  \begin{pmatrix}
    a & b \\
    c & d 
  \end{pmatrix} \in \GL_2(\A) \Big| c \equiv 0 \mod \n \right\}.
  \]
The points on the curve $Y_0(\n)$ parametrize isomorphism classes of
pairs of Drinfeld $\A$-modules of rank two together with a primitive $ \n $-torsion.

 Adding so-called cusps to $ Y_0 (\n)$ gives a projective algebraic curve $X_0(\n) $ defined over $ K $ that in general however will
not be absolutely irreducible. Indeed, the irreducible components are in one-to-one correspondence with the class group $\Cl(\A)$ of $ \A $. 
It is straightforward to see that the absolutely irreducible components of $X_0(\n)$ are mutually isomorphic by considering the action of $\Cl(\A)$.
We will denote by $\x(\n)$ an absolutely irreducible component of $X_0(\n)$. The cusps are distributed equally among the absolutely irreducible components of $X_0(\n)$. In particular, there exist exactly $  |\Cl(\A)| $ cusps on the modular curve $\x(1)$.

To simplify notation, we use the informal definition of  the absolutely irreducible component of $ X_0(\n )$ as follows. 
\begin{defn} \label{defn:equivalence}
      Let $G$ and $H$ be corresponding $\n$-torsion of Drinfeld modules $\phi $ and $\psi$ respectively. 
     We say that the pair $(\phi, G ) $ is equivalent to $ (\psi, H )$, if there exists an isogeny $\lambda \in \bar{L}\{ \tau \}$ from $ \phi $ to $ \psi $ such that  
    $\lambda G = H$. 
\end{defn}
\begin{defn}[Minimal Modular Curve]
 We formally define the ($ \zeta^q$-type) modular curve $ \x (\n)$ as the algebraic curve that parametrizes the pairs $ (\phi, G ) $ modulo the equivalence in Definition \ref{defn:equivalence}, where $ \phi $ denotes the $ \zeta^q$-type Drinfeld $ \A $-modules and $ G $ denotes a primitive $ \n $-torsion of $ \phi $. 
\end{defn}
Analogously, we restrict the Drinfeld modules to be normalized, then we obtain the similar concept of modular curve:
\begin{defn}[Normalized modular curve]
The ($ \zeta^q$-type) normalized modular curve $ \dx (\n)$ is an algebraic curve parametrizing the pairs $ (\phi, G ) $, where $ \phi $ denotes a $ \zeta^q$-type normalized Drinfeld $ \A $-modules and $ G $ denotes a primitive $ \n $-torsion of $ \phi $. 
\end{defn}

In particular, we choose $ \n $ to be of the form $ I_\infty^k $ with $k \geqslant 1 $. Then there are exactly 
\begin{equation}\label{eq:Prim}
  |\Prim_{\n}(\phi)|
  = (q+1) q^ {  k-1 } 
\end{equation}
distinct choices of primitive $ I_\infty^k $-torsion of $ \phi $. For a primitive $I_\infty^{k+1} $-torsion $G$ of $ \phi $, we know 
\[
  I_\infty \cdot G := \{ \phi_z (\mu ) | \mu \in G , z \in I_\infty \}
\]
forms a primitive $I_\infty^{k}$-torsion. Since a geometric point of $ \x (I_\infty^{k+1}) $ is represented by the pair $ (\phi , G)$, the map $(\phi, G ) \mapsto  (\phi , I_\infty \cdot G) $ induces a covering morphism from $ \x(I_\infty^{k+1})  $ to $\x(I_\infty^{k})$. 
So the covering degree of $\x(I_\infty^{k+1}) $ over $ \x(I_\infty^{k}) $ is given by 
\[
   [ \x(I_\infty^{k+1}) : \x(I_\infty^{k})] =  \frac{ | \Prim_{I_\infty^{k+1}}(\phi) | }{ | \Prim_{I_\infty^k}(\phi) |} = \begin{cases}
        q+1, & \text{when $k =0$; } \\
        q,& \text{when $k \geqslant 1$. } 
    \end{cases}
\] 
This formula coincides with degrees displayed in Theorem \ref{thm:thmB}.

\section{Normalized Drinfeld modular tower}\label{sec:thmA}
We wish to construct the normalized Drinfeld modular tower in this section. 
We begin with the construction of isogeny 
\[
    \tau - u_1 : \phi^{\lambda_0} \to \phi^{\sigma; \lambda_1 } 
\]
 for the normalized Drinfeld $\A$-modules. 
\subsection{Isogeny formula}\label{sec:IsogenyFormula}
Suppose that $ \mu_1\in \phi^{\lambda_0} [I_\infty]$, namely
\[
 \phi^{\lambda_0}_x(\mu_1) = 0 \quad \text{and}\quad \phi^{\lambda_0}_y(\mu_1) = 0 .
\]
 
Since the twisted polynomial $\phi_{I_\infty}^{\lambda_0}  $ (see Equation \eqref{eq:phiIinf})  is  the annihilator of $I_\infty$, we have  
\[
\phi^{\lambda_0}_{I_\infty}(\mu_1) = (\tau^2+ \alpha \tau+ \nu \lambda_0^{q-1})(\mu_1) = 0 .
\]
Let $ u_1 =\mu_1^{q-1} $. Then $ u_1 $ satisfies the equality 
\begin{equation}\label{eq:u1lambda0}
 \xi^{\lambda_0} (u_1) := u_1^{q+1}+ \left(\frac{\lambda_0^{q^2}}{1-\zeta^{1-q}}+\frac{\nu }{\zeta T\lambda_0 }\right) u_1 + \nu \lambda_0^{q-1} = 0 . 
\end{equation}

Since $ \tau -u_1 $ is a right-divisor of $ \phi_{x}^{\lambda_0}$ and $ \phi_y^{\lambda_0}$, we know that $\tau - u_1$ is an isogeny from $ \phi^{\lambda_0 } $ to some Drinfeld module of $\zeta$-type, say $ \phi^{\sigma; \lambda_1}$. That is 
\[
    (\tau - u_1)  \phi_a^{\lambda_0 } = \phi_a^{\sigma ; \lambda_1}     (\tau - u_1) 
\]
for all $ a \in \A $. 
The following lemma determines the value of $ \lambda_1 $. 
 
\begin{lem}\label{lem:relation-lam}
    Let $ \phi^{\lambda_0} $ and $\phi^{\sigma; \lambda_1 }: \A \to L\{\tau\} $ denote the normalized Drinfeld $\A$-modules of $\zeta^q$-type and $\zeta$-type respectively. Suppose that $\xi^{\lambda_0} (u_1) =0$.
    Then $ \tau - u_1 $ is an isogeny from $\phi^{\lambda_0} $ to $\phi^{\sigma; \lambda_1}  $, where 
    \begin{equation}\label{eq:lambda1}
    \lambda_1 = \lambda_0^{q}-(\zeta^{q-1}-1)u_1. 
    \end{equation}
\end{lem}
\begin{proof} 
Without loss of generality, we may assume that $ L $ is a perfect field extension over $ \mathbb{F}_{q}(t, \zeta) $ containing both $ \lambda_0 $ and $ \lambda_1 $. 
We define 
\[
V = \{ (a,b) \in L\{ \tau \}^2| \deg_{\tau}(a \phi_y^{\lambda_0} + b \phi_x^{\lambda_0} ) \leqslant 3 , \deg_\tau (a) \leqslant 2, \deg_\tau(b) \leqslant 2 \}.
\]
It is clear that $ V $ is a three-dimensional $L$-vector space and contains the following three vectors by Lemma \ref{lem:phi_I_infty_rep}:
\begin{align*}
    \e_1 &=(1,-\zeta^q),
    \\
    \e_2 &=\left(\frac{\tau+A(\lambda_0)}{C(\lambda_0)}, -\frac{\zeta\tau+B(\lambda_0)}{C(\lambda_0)}\right), \\
    \e_3 &=\tau\left(\frac{\tau+A(\lambda_0)}{C(\lambda_0)}, -\frac{\zeta\tau+B(\lambda_0)}{C(\lambda_0)}\right).
\end{align*} 
Since $\e_1,\e_2,\e_3$ are linearly independent, they form an $L$-basis of $V$. 

By the $\zeta$-type version of Lemma \ref{lem:phi_I_infty_rep}, we have
 \[ \phi^{\sigma; \lambda_1}_{I_\infty} = \left(\frac{\tau+A(\lambda_1 )}{C(\lambda_1 )}\right)^\sigma\phi_y^{\sigma;\lambda_1 } -\left(\frac{\zeta\tau+B(\lambda_1 )}{C(\lambda_1 )}\right)^\sigma\phi_x^{\sigma; \lambda_1 } .
 \]
By assumption, $ \tau - u_1$ is an isogeny, we have
\begin{align*}
  \phi_{I_\infty}^{\sigma;\lambda_1 }(\tau-u_1)= & \left(\frac{\tau+A(\lambda_1 )}{C(\lambda_1 )}\right)^\sigma\phi_y^{\sigma;\lambda_1 }(\tau-u_1)-\left(\frac{\zeta\tau+B(\lambda_1 )}{C(\lambda_1 )}\right)^\sigma\phi_x^{\sigma;\lambda_1 }(\tau-u_1)\\
  = & \left(\frac{\tau+A(\lambda_1 )}{C(\lambda_1 )}\right)^\sigma(\tau-u_1)\phi_y^{\lambda_0}-\left(\frac{\zeta\tau+B(\lambda_1 )}{C(\lambda_1 )}\right)^\sigma(\tau-u_1)\phi_x^{\lambda_0}\\
  = : &  v_1  \phi_y ^{\lambda_0} + v_2 \phi_x ^{\lambda_0}.
\end{align*}
It yields that the pair $\mathbf{v}= (v_1, v_2) \in L\{ \tau \}^2 $ satisfies the required condition of $V$.
Then there exist some coefficients $t_1,t_2,t_3\in L$, such that
\[ \mathbf{v} =t_1 \e_1+t_2 \e_2+t_3 \e_3 ,\]
namely,
\begin{align}
 & \left(  \left(\frac{\tau+A(\lambda_1 )}{C(\lambda_1 )}\right)^\sigma(\tau-u_1),-\left(\frac{\zeta\tau+B(\lambda_1 )}{C(\lambda_1 )}\right)^\sigma(\tau-u_1)\right) \nonumber \\
 &=\left(t_1+(t_2+t_3\tau)\frac{\tau+A(\lambda_0)}{C(\lambda_0)},-\zeta^q t_1-(t_2+t_3\tau)\frac{\zeta\tau+B(\lambda_0)}{C(\lambda_0)}\right) . \label{eq:compare}
\end{align}

We now compare the coefficients of $\tau^i$-terms for $ i=0,1,2 $ in the above equation, proceeding as follows. 
\begin{enumerate}
\item[(1)]   
By comparing the $\tau^2$-terms of \eqref{eq:compare}, we have
\[\left(\frac{1}{C(\lambda_1 )^\sigma},\frac{\zeta^\sigma}{C(\lambda_1 )^\sigma}\right)=\left(t_3C(\lambda_0)^{-q},t_3\zeta^q C(\lambda_0)^{-q}\right).\]
This is equivalent to
\[t_3=C(\lambda_0)^qC(\lambda_1 )^{-\sigma}.\]

\item[(2)]  Coefficients of the $\tau$-terms in \eqref{eq:compare} yield that 
\begin{align*}
   & \left(-\frac{ u_1^q}{C(\lambda_1 )^{\sigma}}+\frac{A(\lambda_1 )^\sigma}{C(\lambda_1 )^{\sigma}},\zeta^\sigma \frac{ u_1^q}{C(\lambda_1 )^{\sigma}}-\frac{B(\lambda_1 )^\sigma}{ C(\lambda_1 )^{\sigma}}\right)\\
    = & \left(\frac{t_2}{C(\lambda_0)}+\frac{t_3A(\lambda_0)^q}{C(\lambda_0)^q},-\frac{t_2\zeta}{C(\lambda_0)}-\frac{t_3 B(\lambda_0)^q}{C(\lambda_0)^q}\right).
\end{align*}
Substituting $t_3=\frac{C(\lambda_0)^q}{C(\lambda_1 )^\sigma}$ into the above equality, equating the first coordinates yields
\[t_2=\frac{C(\lambda_0)}{C(\lambda_1 )^{\sigma}}\left(A(\lambda_1 )^\sigma - u_1^q-A(\lambda_0)^q\right),\]
while equating the second coordinates yields
\[t_2=\frac{C(\lambda_0)}{\zeta C(\lambda_1 )^\sigma}\left(-B(\lambda_0)^q-\zeta^\sigma u_1^q+B(\lambda_1 )^\sigma\right).\]
Equating these two expressions for $t_2$ implies 
\[\zeta\left(A(\lambda_1 )^\sigma - u_1^q-A(\lambda_0)^q\right)=-B(\lambda_0)^q-\zeta^\sigma u_1^q+B(\lambda_1 )^\sigma,
\]
or equivalently,
\[\lambda_1^{q}=\lambda_0^{q^2}-(\zeta^{1-q}-1) u_1^q.\]
Since $L$ is perfect, taking the $q$-th root of both sides gives 
\[\lambda_1=\lambda_0^{q}-(\zeta^{q-1}-1) u_1 . 
\]
\item[(3)]  From the constant terms of \eqref{eq:compare}, it follows 
\[\left(-\frac{A(\lambda_1 )^\sigma u_1}{C(\lambda_1 )^\sigma},\frac{B(\lambda_1 )^\sigma u_1}{C(\lambda_1 )^\sigma}\right)=\left(t_1+t_2\frac{A(\lambda_0)}{C(\lambda_0)},-\zeta^q t_1-t_2\frac{B(\lambda_0)}{C(\lambda_0)}\right).\]
Equating the first coordinates, we obtain
\begin{align*}
t_1=&-\frac{A(\lambda_1 )^\sigma u_1}{C(\lambda_1 )^\sigma}-t_2\frac{A(\lambda_0)}{C(\lambda_0)}\\
=&\frac{1}{C(\lambda_1 )^\sigma}\left(-A(\lambda_1 )^\sigma u_1-A(\lambda_0)A(\lambda_1 )^\sigma+A(\lambda_0) u_1^q+A(\lambda_0)^{q+1}\right).
\end{align*}
Equating the second coordinates, we obtain
\begin{align*}
t_1=&-\frac{B(\lambda_1 )^\sigma u_1}{\zeta^q C(\lambda_1 )^\sigma}-t_2\frac{B(\lambda_0)}{C(\lambda_0)}\\
=&\frac{1}{\zeta C(\lambda_1 )^\sigma}\left(-\zeta^{1-q}B(\lambda_1 )^\sigma u_1+B(\lambda_0)^{q+1}+\zeta^{\sigma}B(\lambda_0) u_1^q-B(\lambda_0)B(\lambda_1 )^\sigma\right).
\end{align*}
Hence, we get
\[t_2=0,\quad t_1=-\frac{A(\lambda_1 )^\sigma u_1}{C(\lambda_1 )^\sigma}.
\]
\end{enumerate}

In conclusion, the isogeny relation between $\phi^{\lambda_0}$ and $\phi^{\sigma;\lambda_1 }$ is
\[
    \lambda_1=\lambda_0^{q}-(\zeta^{q-1}-1)u_1.
\]
\end{proof}
\subsection{Description of the primitive torsions}
We recall the useful lemma in \cite{MR4592575}*{Lemma 3.1.5}. 

\begin{lem}\label{lem:kernel}
    Let $ W $ be an $n$-dimensional $ \mathbb{F}_q$-vector space of $ \bar{K} $, then there exists a unique monic twisted polynomial $ \omega $ of $\tau$-degree $ n $ such that 
    $
        W = \ker \omega  .
    $ 
\end{lem}
If $W$ is identical to a $   I_{\infty}^n $-torsion $ G_n $ of Drinfeld module $\phi$, we obtain a stronger form for $\omega$. 
\begin{lem}\label{lem:primitive}
Given a primitive $   I_{\infty}^n $-torsion $ G_n $ of Drinfeld module $\phi$,  there is a sequence of nonzero elements $ u_1, u_2,\ldots, u_n$ in $ \bar{K}$ such that $ G_n $ is represented as  the kernel of the twisted polynomial
\begin{equation}\label{eq:Omegan}
    \omega_n = (\tau - u_n) \cdots (\tau - u_2 ) (\tau - u_1).
\end{equation}

\end{lem}
\begin{proof}

Fix an integer $ n \geqslant 0$, and set $G_i = I_\infty^{n -i} G_n   $ for $ i = 0, 1, \ldots, n $.
We obtain the sequence 
\[
     0 =  G_0  \subseteq G_1  \subseteq G_2 \subseteq G_3 \subseteq \cdots \subseteq G_n. 
\]
Since $ G_n $ is primitive, there is no doubt that $ G_i   $ is a primitive $I^i_\infty$-torsion of $ \phi $. We know from Lemma \ref{lem:kernel} that $ G_i $ is given by the kernel of some twisted polynomial $\omega_i$. In particular, $ \omega_0 = 1 $. 
Since $ G_i \subsetneq G_{i+1} $, we know $\omega_{i} $ is a right-divisor of $\omega_{i+1} $, i.e., 
\[
\omega_{i+1} = (\tau - u_{i+1}) \omega_i.
\]
Therefore, we have the expression \eqref{eq:Omegan}. 
\end{proof}
Indeed, the variables $ u_1, \ldots, u_n $ can be expressed by the basis of $G_n$. 
Let us fix the primitive $   I_{\infty}^n $-torsion $ G_n $ of a Drinfeld $\A$-module $ \phi $. By Lemma \ref{lem:primitive}, we know there exist $u_1,\ldots, u_n$, such that 
\begin{equation}\label{eq:Gk} 
    G_n = \ker \omega_n = \ker (\tau - u_n) \cdots (\tau - u_2 ) (\tau - u_1). \end{equation}
Suppose that $ \alpha_{1} , \ldots, \alpha_n $ is the basis of $G_n $ such that $ \alpha_{i} $ is contained in $ G_{i} \setminus G_{i-1} $. Then $\alpha_i $ satisfies 
\[
    \omega_{i}(\alpha_{i} ) = (\tau - u_{i} ) \omega_{i-1} (\alpha_{i})= 0 
\]
and 
\[
\omega_{i-1}(\alpha_{i} ) \not= 0 .
\]
This implies that 
\begin{equation}\label{eq:uiomega}
    u_{i} = \omega_{i-1}(\alpha_{i} )^{q-1}. 
\end{equation}

\subsection{Iterative formula of isogeny}
Let $ u_i , \alpha_i , \omega_i $ be defined as above associated to the primitive $I_\infty$-torsion $G_n$ of the  $\zeta^{q}$-type normalized Drinfeld module $\phi^{\lambda_0}$. 
\begin{lem}\label{lem:isogeny}
  Maintain the notations above. Suppose that $ \lambda_k $ with $  1 \leqslant k \leqslant n$ is defined by the iterative formula
     \begin{equation}\label{eq:lambdaiui}
           \lambda_k = \lambda_{k-1}^q-(\zeta^{q^k-{q^{k-1}}}-1) u_k.
     \end{equation}
      Then   
     \begin{equation}\label{eq:phivanish}
        \phi^{\sigma^{k-1}; \lambda_{k-1}}_{I_\infty}(\omega_{k-1 } (\alpha_{k})) =0  ,
     \end{equation}
     or equivalently, $u_k $ and $ \lambda_{k-1}$ satisfy 
      \begin{equation}\label{eq:xi}
     \xi^{\sigma^{k-1}; \lambda_{k-1}} (u_{k}) := u_{k}^{q+1}+ \left(\frac{\lambda_{k-1}^{q^2}}{ 1-\zeta^{q^{k-1}-q^{k}} }+\frac{\nu^{\sigma^{k-1}} }{\zeta^{q^{k-1}} T^{\sigma^{k-1}}\lambda_{k-1} }\right) u_{k} + \nu^{\sigma^{k-1}}  \lambda_{k-1}^{q-1} = 0 . 
      \end{equation}
      Moreover, the twisted polynomial $\omega_k$ represents an isogeny from $ \phi^{\lambda_0} $ to $ \phi^{\sigma^k; \lambda_k} $, namely,
     \begin{equation}\label{eq:Omega_k}
    \omega_k \phi^{\lambda_0} = \phi^{\sigma^k; \lambda_k} \omega_k . 
     \end{equation}
   
\end{lem}
\begin{proof}
    For the case $ k =1 $, the equality \eqref{eq:phivanish} just says that $ \alpha_1 $ is contained in $G_1$, which has been investigated in Section \ref{sec:IsogenyFormula}. Since $u_1 = \omega_0(\alpha_1)^{q-1} = \alpha_1^{q-1}$, we know from Section \ref{sec:IsogenyFormula} that $u_1$ and $\lambda_0 $ automatically satisfy Equation \eqref{eq:u1lambda0}, i.e., $\xi^{\lambda_0}(u_1) =0 $. Equation \eqref{eq:Omega_k} with $k=1$ has been shown in Lemma \ref{lem:relation-lam}.

     Moreover, we can deduce the equivalence of \eqref{eq:phivanish} and \eqref{eq:xi} by the equality \eqref{eq:uiomega}
     for each $k \leqslant n $.

Next, we assume by induction that the lemma  is valid for $ 1\leqslant k \leqslant i $. It suffices to prove the case $ k = i+1 $.  
By defintion, $ G_{i+1 } $ is a primitive $ I_{\infty}^{i+1} $-torsion of $ \phi^{\lambda_0}$.  For $ z \in I_\infty$, we have 
\begin{equation}\label{eq:GiCon1}
\phi_{z}^{\lambda_0} G_{i+1} \subseteq G_{i} .
\end{equation}
From Equation \eqref{eq:Gk}, it yields that $ \phi_{z}^{\lambda_0}(\alpha_{i+1}) $ is always annihilated by $ \omega_i $. Equation \eqref{eq:Omega_k} with $k = i$ implies that 
\begin{equation}\label{eq:lambdaCon1}
0 = \omega_{i}(\phi_{z}^{\lambda_0}(\alpha_{i+1}) ) =\phi_{z}^{\sigma^i; \lambda_i}(\omega_{i}\alpha_{i+1})  \quad \text{for all $z \in I_\infty $}. 
\end{equation}
Note that $\phi_{I_\infty}$ is the annihilator of ideal $I_\infty$. Thus, \eqref{eq:lambdaCon1} is equivalent to 
\[
 \phi_{I_\infty}^{\sigma^i; \lambda_i}(\omega_{i}\alpha_{i+1}) = 0. 
\]
This verifies the Equation \eqref{eq:phivanish} with $ k = i+1 $.

On the other hand, applying the $k$-th $\sigma$-action in Lemma \ref{lem:relation-lam}, the relation of $\lambda_i$ and $\lambda_{i+1}$ gives that
\begin{equation}\label{eq:lambdai}
(\tau-u_{i+1})\phi^{\sigma^{i};\lambda_{i}}=\phi^{\sigma^{i+1};\lambda_{i+1}}(\tau-u_{i+1}).
\end{equation}
Therefore,
\begin{align*}
\omega_{i+1}\phi^{\lambda_0}=&(\tau-u_{i+1}) \omega_i \phi^{\lambda_0}\\
=&(\tau-u_{i+1})\phi^{\sigma^i; \lambda_i} \omega_i \\
=&\phi^{\sigma^{i+1};\lambda_{i+1}} (\tau-u_{i+1})\omega_i  \quad \text{ By \eqref{eq:lambdai}}
\\
=&\phi^{\sigma^{i+1};\lambda_{i+1}}  \omega_{i+1}.
\end{align*}
The equality \label{eq:Omegak} holds for $ k = i +1 $. 
\end{proof}

\subsection{Restrictions on primitive $ I^{i+1}_\infty$-torsion}\label{sec:Restriction}
Next, we need to derive the restriction on the variables $u_{i+1}$ from the information on $ G_n $, a primitive $ I_{\infty}^{i+1} $-torsion of $ \phi^{\lambda_0} $.
In our assumption, $ G_{i+1 } $ shall be a primitive $ I_{\infty}^{i+1} $-torsion. It splits into two conditions:
\begin{enumerate}
  \item  For $ z \in I_\infty$, $\phi_{z}^{\lambda_0}( G_{i+1})$ is contained in $ G_{i}$, i.e., Equation \eqref{eq:lambdaCon1} holds;
 \item $ \phi_{z}^{\lambda_0} (G_{i+1})$ is not always contained in $G_{i-1}$. 
\end{enumerate}

From the proof of lemma \ref{lem:isogeny}, we know 
the first condition is equivalent to
\[ 
\xi^{\sigma^i; \lambda_i} (u_{i+1})=0 .
\]

The second condition requires that $ \phi_{z}^{\lambda_0} (\alpha_{i+1}) $ is not always annihilated by $ \omega_{i-1} $. Equivalently, 
\begin{equation}\label{eq:NotZero}
\phi_{I_\infty}^{\sigma^{i-1}; \lambda_{i-1}} \omega_{i-1}(\alpha_{i+1}) \neq  0.
\end{equation}
The equality \eqref{eq:phivanish} with $k= i-1$ says that $\omega_{i-1}(\alpha_{i})  $ is annihilated by $ \phi_{I_\infty}^{\sigma^{i-1}; \lambda_{i-1}} $. 
This yields that $(\tau-u_i)$ is a right-divisor of $ \phi_{I_\infty}^{\sigma^{i-1}; \lambda_{i-1}} $. In fact, we obtain the decomposition
\[ 
\phi_{I_\infty}^{\sigma^{i-1};\lambda_{i-1}}=\left(\tau- u_i^{\nabla}  \right)(\tau-u_i),
\]
where $u_i^{\nabla} := \frac{\nu^ {\sigma^{i-1}}\lambda_{i-1}^{q-1}}{u_i}. $
Thus, the left hand side of \eqref{eq:NotZero} equals 
\begin{align*}
  \phi_{I_\infty}^{\sigma^{i-1}; \lambda_{i-1}} \omega_{i-1}(\alpha_{i+1})&= \left(\tau-u_i^{\nabla} \right)(\tau-u_i)(\omega_{i-1}(\alpha_{i+1}))\\
  & =\left(\tau-u_i^{\nabla} \right) (\omega_{i} (\alpha_{i+1}))\\
  & = \omega_{i} (\alpha_{i+1})\cdot \left( u_{i+1}- u_i^{\nabla} \right).
\end{align*}
Since $\omega_{i} (\alpha_{i+1})\neq 0$,  Equation \eqref{eq:NotZero} is the same as 
$
u_{i+1} \neq u_i^{\nabla}.
$
From the argument above, we see that $u_{i+1} =u_i^{\nabla}$ must be a root of $ \xi^{\sigma^i; \lambda_i} (u_{i+1})$. 
Indeed, it is straightforward to check that polynomial $\xi^{\sigma^i; \lambda_i}(u_{i+1})$ admits a decomposition
\begin{equation*}
   \xi^{\sigma^i; \lambda_i}(u_{i+1})
    = \xi_\nabla^{ \sigma^i; \lambda_i} (u_{i+1}) \cdot \left(u_{i+1}- u_i^{\nabla}\right) ,
\end{equation*}
where 
\begin{align}
     \xi_\nabla^{ \sigma^i; \lambda_i} (u_{i+1})  = &  \frac{\lambda_i^{ q^2}}{1-\zeta^{(1-q)q^i}}+\frac{\nu^{\sigma^{i}}}{T^{\sigma^i} \lambda_i\zeta^{q^i}}+\sum_{s=0}^{q} \left(u_i^{\nabla}\right)^s u_{i+1}^{q-s} \nonumber \\
 = &-\lambda_i^{q-1}\frac{\nu^{\sigma^i} u_i}{\nu^{\sigma^{i-1}}\lambda_{i-1}^{q-1}}+\sum_{s=0}^{q-1} \left(u_i^{\nabla}\right)^{s}u_{i+1}^{q-s} \label{eq:exlambdai}.
\end{align}
In conclusion, the two conditions on $G_{i+1}$ together are equivalent to 
\[
 \xi_\nabla^{ \sigma^i; \lambda_i} (u_{i+1}) =0 .
\]

\subsection{Proof of Theorem \ref{thm:thmA}}
Now we are able to state the proof of Theorem \ref{thm:thmA}.
\begin{proof}[Proof of Theorem \ref{thm:thmA}]
For the case $k=0 $, the theorem just says that $\lambda_0 $ represents the complete family of $\zeta^q$-type normalized $\A$-Drinfeld module, which has been investigated in Theorem \ref{thm:normalized}. 

For $k \geqslant 1 $, from the definition of Drinfeld modular curves, it suffices to show that 
$ \lambda_0, \ldots , \lambda_k $ is in one-to-one correspondence with primitive $I_\infty^k$-torsions of $ \phi^{\lambda_0 } $. 

The case $ k =1 $ can be easily checked that $ (\lambda_0, u_1) $ corresponds to the  primitive $I_\infty $-torsion
\[
    G_1 = \ker(\tau - u_1) 
\]
where $ u_1 $ verifies $ \xi^{\lambda_0}(u_1) = 0 $. 
Indicated by Equation \eqref{eq:lambda1}, we choose transformation:
\[
    u_1 := \frac{\lambda_0^q - \lambda_1}{\zeta^{q-1}-1}.
\]

Substituting $ u_1 $ into $\xi^{\lambda_0}$ it yields that 
\[
\xi^{\lambda_0} \left(\frac{\lambda_0^q - \lambda_1}{\zeta^{q-1}-1}\right) =0 , 
\]
i.e., 
\[\lambda_1^{q+1}-\lambda_0^{q}\lambda_1^{ q}-\frac{\zeta^{1-q}-1}{\zeta T\lambda_0}\nu\lambda_1+\left(\frac{\zeta^{1-q}-1}{\zeta T}+(\zeta^{q-1}-1)^{q+1}\right)\nu\lambda_0^{q-1}=0.
\]
So $ (\lambda_0, \lambda_1) $ is in one-to-one correspondence with primitive $ I_\infty $-torsion.

Next, we assume that $ k\geqslant 2 $. Given the primitive $I_\infty^{k}$-torsion $G_k$, we construct $ u_i, \lambda_i $ as above. In Section \ref{sec:Restriction}, we have realized that it is equivalent to 
\begin{equation}\label{eq:xinabla}
 \xi_\nabla^{\sigma^i; \lambda_i }(u_{i+1}) = 0,  
\end{equation}
for $i =0, \ldots, k-1$. 
It follows from the equality \eqref{eq:lambdaiui} of Lemma \ref{lem:isogeny} that   
 \[
u_{i+1} :=\frac{\lambda_{i}^q-\lambda_{i+1} }{ \zeta^{q^{i+1}-q^{i}}-1   }
\]
and 
\[
u_i :=\frac{\lambda_{i-1}^q-\lambda_i }{ \zeta^{q^{i}-q^{i-1}}-1  }.
\]
Combining these with Equation \eqref{eq:xinabla},  we have 
\begin{align}
    \xi_\nabla^{\sigma^i; \lambda_i }(u_{i+1}) & =-\lambda_i^{q-1}\frac{\nu^{\sigma^i} (\lambda_{i}-\lambda_{i-1}^{q})}{\nu^{\sigma^{i-1}}\lambda_{i-1}^{q-1}} \nonumber \\
    & +\sum_{s=0}^{q-1} \left(\frac{\nu^{\sigma^{i-1}}\lambda_{i-1}^{q-1}(1-\zeta^{1-q})^{q+1}}{ \lambda_{i}-\lambda_{i-1}^{q} }\right)^{s}(\lambda_{i+1}-\lambda_{i}^{q})^{q-s} =0 .\label{eq:conlambda}
\end{align}

Conversely, we can reconstruct the  $I_\infty^i$-primitive torsions $G_i$ through $ \lambda_i $ as follows:
\begin{align*}
    G_i &= \ker(\omega_i ) \\
        &= \ker \left(\tau -\frac{\lambda_{i}-\lambda_{i-1}^{q}}{1-\zeta^{(1-q)q^{i-2}}} \right)\left(\tau -\frac{\lambda_{i-1}-\lambda_{i-2}^{q}}{1-\zeta^{(1-q)q^{i-3}}}\right)\cdots\left(\tau- \frac{\lambda_{1}-\lambda_{0}^{q}}{1-\zeta^{(1-q)q^{-1}}} \right)
\end{align*}
where $\lambda_0, \ldots, \lambda_i $ verifies the recursive condition \eqref{eq:conlambda}. 
\end{proof}

\section{Minimal Drinfeld modular tower}\label{sec:thmB}

\subsection{The modular curve $ \x(I_\infty) $}\label{sec:modularCurve1}
We start with a minimal Drinfeld module $ \Phi^{j_0} $ parameterized by the $j$-invariant $ j_0 $. We know from Section \ref{sec:minial} that a $(q-1)$-th root of $ \lambda_0^q $ gives an isogeny from $ \Phi^{j_0} $ to $ \phi^{\lambda_0 }$, 
where $\lambda_0$ satisfies
\begin{equation}\label{eq:j0} j_0 =\frac{ \lambda_0^{q^2+1}  }{\nu }. 
\end{equation} 
Suppose that $ \beta_1 \not= 0 $ is contained in $ \Phi^{j_0}[I_\infty] $, i.e., $ \Phi^{j_0}_{I_\infty}(\beta_1) = 0 $. Then the space $ G_1:=\mathbb{F}_{q} \beta_1$ is a primitive $I_\infty$-torsion of $ \Phi^{j_0} $. From the isogeny $\ell_0$, we find that the element $ \alpha_1 := \ell_0 \beta_1 $ is annihilated by $ \phi_{I_\infty} $, and hence 
\begin{equation}\label{eq:PhiIinf}
    \Phi_{I_\infty}^{j_0} = \ell_0^{-q^2} \phi_{I_\infty}^{\lambda_0} \ell_0.
\end{equation}
Substituting the expression \eqref{eq:phiIinf}, we obtain 
\begin{equation}\label{eq:jinv}
      \Phi_{I_\infty}^{j_0} = \tau^2+\left(\frac{1}{1-\zeta^{1-q}}+\frac{1}{\zeta T}\frac{1}{j_0}\right)\tau +\frac{1}{j_0}  . 
\end{equation}
This equality can also be deduced from the expression of $\Phi^{j_0}$, as $\Phi_{I_\infty}^{j_0} $ is exactly the right-divisor of both $\Phi_x^{j_0}$ and $ \Phi_y^{j_0}$.

Take $ w_1 = \beta_1^{q-1}$. Then applying \eqref{eq:jinv} yields that 
\begin{equation}\label{eq:w1}
    \Xi^{j_0}(w_1):= w_1^{q+1} + \left(\frac{1}{1-\zeta^{1-q}}+\frac{1}{\zeta T}\frac{1}{j_0}\right)w_1 +\frac{1}{j_0}  =\frac{1}{\beta_1}\Phi_{I_\infty}^{j_0}(\beta_1)= 0 . 
\end{equation}

Since $G_1 $ can be expressed as 
$
    \ker(\tau - w_1)
$, we find that the function field of the modular curve $ \x(I_\infty) $ is  $ H (j_0,w_1) $, 
where $ j_0, w_1 $ are subject to \eqref{eq:w1}. 
Furthermore, from the equation \eqref{eq:w1}, the $j$-invariant of $ \Phi^{j_0} $ is determined by $w_1$ of the form  
\begin{equation}\label{eq:jw}
    j_0=\frac{-(1+\zeta^{-1}T^{-1} \beta_1^{q-1})}{\beta_1^{q^2-1}+(1-\zeta^{1-q})^{-1}\beta^{q-1}}=\frac{( \zeta^{1-q}-1)(1+\zeta^{-1}T^{-1}w_1)}{ w_1\left( 1+ (1-\zeta^{1-q})w_1^{q} \right) }.
\end{equation}
Thus the function field of the modular curve $ \x(I_\infty) $ is rational and is generated by $ w_1 $.


\subsection{Isogeny formula for  minimal Drinfeld modules}
In the rest, we always assume that $w_1 $ and $j_0$ satisfy the equality \eqref{eq:jw}. 
We have seen that the isogeny relation 
\begin{equation}\label{Eq:isogeny2}
(\tau-u_1)\phi^{\lambda_0}=\phi^{ \sigma; \lambda_1}(\tau-u_1)
\end{equation}
is essential in the construction of normalized Drinfeld modular tower. We now derive the analogous relation for $ \Phi^{j_0}$. 
\begin{lem}\label{lem:isogenyPhi}
    Suppose that $w_1 $ is a root of the equation $ \Xi^{j_0}(w_1) = 0 $.
    Let $\delta_1 $ be a $(q-1)$-th root of the constant $(1 - (\zeta^{1-q}-1)w_1^q )^{-1}$. Then the twisted polynomial 
    \[ \Omega_1:=\delta_1 (\tau - w_1)
    \]
     represents an isogeny from $ \Phi^{j_0} $ to $ \Phi^{\sigma; j_1 }$, where $ j_1 $ is determined by 
    \begin{align}
        j_1 &=  T^{q-1} j_0^q \left( 1 +(1-\zeta^{q-1})   w_1 \right)^{q^2+1}  \label{eq:j0j1} \\
     & = T^{-1} (\zeta^{q-1}-1)   \left( 1 +(1-\zeta^{q-1})   w_1 \right)   (T^q w_1^{-q}+\zeta^{-q} ) .\label{eq:j0w1}
    \end{align}
\end{lem}
\begin{proof}
Let $ \phi^{\lambda_0} $ and $ \phi^{\sigma; \lambda_1} $ be the same notation as in Equation \eqref{Eq:isogeny2}.
Let $ \ell_0 $ be a $(q-1)$-th root of $\lambda_0^q$; and $ \ell_1 $ a $(q-1)$-th root of $\lambda_1^q$. Let $ j_1 $ be the $j$-invariant of $ \phi^{\sigma; \lambda_1}$, i.e., 
 \[
    j_1 = \frac{\lambda_1^{q^2+1}}{\nu^{\sigma}}.
\]
We obtain the diagram of square consisting of isogenies: 
\begin{equation}\label{eq:isogeny}
    \begin{tikzcd}
        \phi^{\lambda_0} \ar[r, "\tau - u_1"]& \phi^{\sigma;\lambda_1}\\
        \Phi^{j_0} \ar[r]\ar[u, "\ell_0"]  &  \Phi^{\sigma;j_1} \ar[u, "\ell_1"]  
    \end{tikzcd}
\end{equation}
As in Section \ref{sec:modularCurve1}, we choose $ \beta_1 $ to be a $ (q-1) $-th root of $w_1$.
Note that $ \beta_1 \in \Phi^{j_0}[I_\infty] $, and $ \alpha_1 := \ell_0 \beta_1 $ is then contained in the $I_\infty$-torsion of $\phi^{\lambda_0 } $. 
By the assumption $ u_1 = \alpha_1^{q-1} $ and $w_1 = \beta_1^{q-1} $, we obtain 
\[ 
w_1 \ell_0^{q-1}= u_1.  
\]  
From the diagram \eqref{eq:isogeny}, the isogeny from $\Phi^{j_0}$ to $\Phi^{\sigma;j_1}$ can be written as 
\[ 
 \Omega_1 := \ell_1^{-1} (\tau - u_1 )\ell_0 = \ell_1^{-1}\ell_0^q(\tau-w_1)= \delta_1 (\tau-w_1),
\]
where $ \delta_1 = \ell_1^{-1}\ell_0^q$.
 Indeed, this can be checked directly:
\begin{align*}
   \delta_1 (\tau-w_1)\Phi^{j_0}&=\ell_1^{-1}(\tau-u_1) \phi^{\lambda_0}\ell_0 \quad\text{ By \eqref{eq:PhiIinf}}\\
    &=\ell_1^{-1}\phi^{\sigma;\lambda_1}(\tau-u_1)\ell_0 \quad \text{By \eqref{Eq:isogeny2}} \\
    &= \Phi^{\sigma;j_1}\ell_1^{-1}(\tau-u_1)\ell_0 \quad \text{Analogous to \eqref{eq:PhiIinf}} \\
    &= \Phi^{\sigma;j_1}\ell_1^{-1}\ell_0^q(\tau-w_1)\\
    &= \Phi^{\sigma;j_1}\delta_1(\tau-w_1).
\end{align*}

Next, we need to compute $ \delta_1   $ and $j_1$.  
From the relation 
\[
\lambda_1=\lambda_0^{q}-(\zeta^{q-1}-1)u_1 = \lambda_0^{q} \left( 1 -(\zeta^{q-1}-1)   w_1 \right)
\]
we have  
\begin{equation}\label{eq:ellell}
    \delta_1^{q-1} = (\ell_0^q\ell_1^{-1})^{q-1} = \lambda_0^{q^2} \lambda_1^{-q}  = \frac{1}{ 1 + (1-\zeta^{1-q})w_1^q }.
\end{equation}

It follows from the expression $ \nu^{\sigma}$ in \eqref{eq:nusigma} that 
\begin{align*}
    j_1 & =\frac{\lambda_1^{q^2+1}}{\nu^{\sigma}} = \frac{1}{\nu^\sigma}  \left(\lambda_0^{q} \left( 1 +(1-\zeta^{q-1})   w_1 \right) \right)^{q^2+1} \nonumber \\
    & = \frac{T^{q-1}\lambda_0^{q(q^2+1)}}{\nu^q}\left( \left( 1 +(1-\zeta^{q-1})   w_1 \right) \right)^{q^2+1}\nonumber \\
      & =   T^{q-1} j_0^q \left( \left( 1 +(1-\zeta^{q-1})   w_1 \right) \right)^{q^2+1},
\end{align*}
where we apply Equation \eqref{eq:j0} in the last equality. So the equality \eqref{eq:j0j1} is valid.

Set  $ w_1^{\nabla }: = (\zeta^{1-q}-1)^{-1}(1+\zeta^{-1}T^{-1}w_1)^{-1} $. From Equation \eqref{eq:jw}, we have 
\[
  \frac{\delta_{1}^{q-1}}{w_{1}^{\nabla} w_{1}} = \frac{1}{ \left( 1+ (1-\zeta^{1-q})w_1^q \right)  w_{1} \cdot w_{1}^{\nabla}  } = j_0.
   \]
Substituting this into the expression \eqref{eq:j0j1}, we have  
\begin{align*}
j_1 &=  T^{q-1} j_0^q \left(   1 +(1-\zeta^{q-1})   w_1 \right)  ^{q^2+1}\\
& = T^{ q-1} \frac{  \left( 1 +(1-\zeta^{q-1})   w_1^{q^2} \right)    \left( 1 +(1-\zeta^{q-1})   w_1 \right)}{ 1+ (1-\zeta^{q-1})w_1^{q^2}    }  \left(\frac{1}{ w_1 w_1^{\nabla}} \right) ^q \\
& = \frac{T^{ q-1}    \left( 1 +(1-\zeta^{q-1})   w_1 \right)  }{\left( w_1 w_1^{\nabla}\right)^q }\\
& =   T^{  -1} (\zeta^{q-1}-1)   \left( 1 +(1-\zeta^{q-1})   w_1 \right)   (T^q w_1^{-q}+\zeta^{-q} ) .
\end{align*}
This is exactly the equality \eqref{eq:j0w1}.
\end{proof}

\subsection{Restrictions on primitive $I_\infty^n$-torsion}\label{sec:RestrictionPhi}
To show the case $ n\geqslant 2 $ of Theorem \ref{thm:thmB}, it is essential to choose coordinates $w_1,\ldots, w_n$ to parametrize the primitive $I_\infty^n$-torsions of $ \Phi^{j_0} $.
Suppose that the elements $ \beta_1 , \beta_2 ,\ldots , \beta_{n} $ span a primitive  $I_\infty^n $-torsion of $ \Phi^{j_0}$.
Without loss of generality, we assume that $ \beta_k \in \Phi^{j_0}[I_\infty^k] \setminus \Phi^{j_0}[I_\infty^{k-1} ] $ for $ 1 \leqslant k \leqslant n$.
We define recursively the twisted polynomial $\Omega_k $, with $ 0 \leqslant k \leqslant n$ as follows.
We set $ \Omega_0 = 1 $ and
\begin{align}
\Omega_k & =  \delta_k (\tau - w_{k}) \Omega_{k-1} 
\end{align}
where 
\begin{equation}\label{eq:wk}
    w_k := \Omega_{k-1} (\beta_k)^{q-1}, 
\end{equation}
and $\delta_k $ is a $(q-1)$-th root of  
\[ 
 (1 + (1-\zeta^{q^{k-1}-q^{k}})w_k^q )^{-1} . 
\]
 It is straightforward to show that 
 \begin{align}
\Omega_k & =  \delta_k(\tau - w_{k})\cdots \delta_2(\tau - w_2)\delta_1(\tau - w_1) \\
& = \delta_1\delta_2\cdots \delta_k ((\delta_1\delta_2\cdots \delta_{k-1})^{q-1}\tau - w_{k})\cdots   (  (\delta_2\delta_1)^{q-1} \tau - w_3 )  (\delta_1^{q-1} \tau - w_2 )(\tau - w_1).
\end{align}
Adopting a similar proof as in Lemma \ref{lem:isogeny}, we have the following lemma.
\begin{lem}\label{lem:Omega}
    Maintain the notations above. Then 
    \begin{enumerate}
    \item The elements 
$\beta_1 , \beta_2 ,\ldots , \beta_{k}$ are annihilated by $ \Omega_k $.  
\item The element $\Omega_{k-1} (\beta_k)$ is annihilated by $ \Phi_{I_\infty}^{\sigma^{k-1}; j_{k-1}} $, and thus
\begin{equation}\label{eq:DecPhi}
\Phi^{\sigma^{k-1};j_{k-1}}_{I_\infty}  =  \left(\tau - \frac{1}{w_k j_{k-1}}\right)\left(\tau - w_k\right).
\end{equation}
\item The twisted polynomial $\delta_k (\tau - w_k)$ is an isogeny from $ \Phi^{\sigma^{k}; j_k} $ to $ \Phi^{\sigma^{k+1}; j_{k+1}} $, where $ j_k $ are recursively defined as 
\begin{equation}\label{eq:jk}
     j_k = j_k(w_k)= T^{  -\sigma^{k-1}} (\zeta^{q^k-q^{k-1}}-1)   \left( 1 +(1-\zeta^{q^k-q^{k-1}})   w_k \right)   (T^{q \sigma^{k-1}} w_k^{-q}+\zeta^{-q^{k}} ) .
\end{equation}
\item The twisted polynomial $ \Omega_k $ is indeed an isogeny from $ \Phi^{j_0} $ to $ \Phi^{\sigma^k; j_k}$.
    \end{enumerate}
\end{lem}
For $1 \leqslant k \leqslant n-1$, the restriction on $\beta_{k+1} $ is that for $z \in I_\infty$, 
\begin{equation}\label{eq:Con1} 
    \Phi_{z}^{j_0}(\beta_{k+1}) \in \langle \beta_1, \ldots, \beta_{k}\rangle
\end{equation}
and there exists some $ z_0 $ such that 
\begin{equation} \label{eq:Con2}
    \Phi_{z_0}^{j_0} (\beta_{k+1})  \not\in \langle \beta_1, \ldots, \beta_{k-1}\rangle . 
\end{equation}
From Lemma \ref{lem:Omega}, the condition \eqref{eq:Con1} is equivalent to 
\begin{equation}\label{eq:Con1ex} 
0 = \Omega_{k}\Phi_{z}^{j_0}(\beta_{k+1}) = \Phi_{z}^{\sigma^{k}; j_{k}} \left( \Omega_{k}(\beta_{k+1})\right) 
\end{equation}
for all $ z \in I_\infty $. Applying $\sigma^{k}$ to \eqref{eq:jinv}, we know the annihilator of the ideal $I_\infty$ is given by 
\[
  \Phi_{I_\infty}^{\sigma^{k}; j_{k}} = \tau^2+\left(\frac{1}{1-\zeta^{q^{k}-q^{k+1}}}+\frac{1}{\zeta^{q^{k}} T^{\sigma^{k}}}\frac{1}{j_{k}}\right)\tau +\frac{1}{j_{k}}  . 
\]
Define the polynomial $ \Xi^{\sigma^{k}; j_{k}}(w_{k+1}) $ in the variable $w_{k+1} $ as 
\[
 \Xi^{\sigma^{k}; j_{k}}(w_{k+1}) :=  w_{k+1}^{q+1}+\left(\frac{1}{1-\zeta^{q^{k}-q^{k+1}}}+\frac{1}{\zeta^{q^{k}} T^{\sigma^{k}}}\frac{1}{j_{k}}\right)w_{k+1} +\frac{1}{j_{k}}  
\]
where $ j_k $ is subject to Equation \eqref{eq:jk}.
Recall that $ w_{k+1} =  \Omega_k (\beta_{k+1})^{q-1}$, we derive that \eqref{eq:Con1ex} (resp. \eqref{eq:Con1}) is equivalent to 
\[
 \Xi^{\sigma^{k}; j_{k}}(w_{k+1}) = \frac{ \Phi_{I_\infty}^{\sigma^{k}; j_{k}} \left(\Omega_k (\beta_{k+1})\right) }{  \Omega_k (\beta_{k+1}) } = 0. 
\]
On the other hand, Lemma \ref{lem:Omega} and the condition \eqref{eq:Con2} yield that for $z \in I_\infty$, 
\begin{equation*} 
 \Omega_{k-1}\Phi_{z}^{j_0}(\beta_{k+1}) = \Phi_{z}^{\sigma^{k-1}; j_{k-1}} \left( \Omega_{k-1}(\beta_{k+1})  \right)
\end{equation*}
does not always vanish.  Thus, we derive the inequality 
\begin{equation}\label{eq:neq}
 0 \neq \Phi_{I_\infty}^{\sigma^{k-1}; j_{k-1}} \left( \Omega_{k-1}(\beta_{k+1})  \right).
\end{equation}
Applying the decomposition \eqref{eq:DecPhi}
we get
\begin{align*}
   \Phi_{I_\infty}^{\sigma^{k-1}; j_{k-1}} \left( \Omega_{k-1}(\beta_{k+1})\right) &= \left(\tau - \frac{1}{j_{k-1} w_{k}}\right)(\tau - w_{k})\left( \Omega_{k-1}(\beta_{k+1})\right) \\
   &=\delta_{k}^{-q} \left(\tau - \frac{\delta_{k}^{q-1}}{j_{k-1} w_{k}}\right) \delta_{k} (\tau - w_{k})\left( \Omega_{k-1}(\beta_{k+1})\right)\\
   &=\delta_{k}^{-q} \left(\tau - \frac{\delta_{k}^{q-1}}{j_{k-1} w_{k}}\right) \left(\Omega_{k}(\beta_{k+1}) \right)\\
   &= \left(\Omega_{k}(\beta_{k+1}) \right)\delta_{k}^{-q} \left(w_{k+1} - \frac{\delta_{k}^{q-1}}{j_{k-1} w_{k}}\right). 
\end{align*}
The inequality \eqref{eq:neq} yields that $w_{k+1} \neq w_{k}^{\nabla}$, where 
\[
    w_{k}^{\nabla} = \frac{\delta_{k}^{q-1}}{j_{k-1} w_{k}} . 
\]
From the equality
\[\Xi^{\sigma^{k-1}; j_{k-1}} (w_k) =  w_{k}^{q+1}+\left(\frac{1}{1-\zeta^{q^{k-1}-q^{k}}}+\frac{1}{\zeta^{q^{k-1}} T^{\sigma^{k-1}}}\frac{1}{j_{k-1}}\right)w_{k} +\frac{1}{j_{k-1}}  = 0 ,
\] 
 we get 
\[
j_{k-1}=   \frac{ \left(  \zeta^{q^{k-1}-q^{k}} -1 \right) \left( w_{k} \zeta^{-q^{k-1}} T^{-\sigma^{k-1}}  +1  \right) }{
    \left( 1-\zeta^{q^{k-1}-q^{k}} \right)w_{k}^{q+1}+  w_{k}  }.
\]
So 
\[
    w_{k}^{\nabla} = \frac{\delta_{k}^{q-1}}{j_{k-1} w_{k}} = \frac{
    1  }{   \left(  \zeta^{q^{k-1}-q^{k}} -1 \right) \left( w_{k} \zeta^{-q^{k-1}} T^{-\sigma^{k-1}}  +1  \right) }.
\]
Remember that $w_{k+1} = w_{k}^{\nabla}$ must be a special root of 
$ \Xi^{\sigma^k ; j_k } (w_{k+1} ) = 0 $. There must be a decomposition 
   \[  \Xi^{\sigma^k;j_k} (w_{k+1})= \Xi_{\nabla}^{\sigma^k; j_k}(w_{k+1}) \left(w_{k+1} - w_k^\nabla \right) \]
   for some polynomial $\Xi_{\nabla}^{\sigma^k; j_k}(w_{k+1}) $.
   It is straightforward to show that  
    \begin{align*}
        \Xi_{\nabla}^{\sigma^k; j_k}(w_{k+1})=  -\frac{w_k^q}{1-(\zeta^{q^k-q^{k-1}}-1)w_k  } \left(\frac{w_k^\nabla}{T^{\sigma^{k-1}}} \right)^{q-1} +\sum_{i=0}^{q-1}  (w_k^\nabla)^{i}w_{k+1}^{q-i}.
    \end{align*}
In conclusion, we understand that the restrictions \eqref{eq:Con1} and \eqref{eq:Con2} on $ \beta_{k+1} $ reduce to the equality 
\[ \Xi_\nabla^{\sigma^k ; j_k } (w_{k+1} ) = 0.
\]

\subsection{Proof of Theorem \ref{thm:thmB}}

    The case $ n =0 $ of Theorem \ref{thm:thmB} is the restatement of \ref{thm:Drinfeld}. 
In Section \ref{sec:modularCurve1}, we have proved the case $ n =1 $.
\begin{proof}[Proof of Theorem \ref{thm:thmB}]
Assume that a primitive $I_\infty^n $-torsion of $ \Phi^{j_0} $ is spanned by $ \beta_1, \ldots , \beta_n $. Let $w_1, \ldots, w_n $ be the variables associated to $ \beta_1,\cdots , \beta_n $ as in Equation \eqref{eq:wk}. We conclude from Section \ref{sec:modularCurve1} that $w_1 $ satisfies the equation \[ \Xi^{j_0}(w_1) = 0 .\] 
For $ k =2 ,\ldots, n-1$, Section \ref{sec:RestrictionPhi} implies that $w_k$ satisfies  
\[
\Xi_{\nabla}^{\sigma^{k-1}; j_{k-1}}(w_{k}) =0 .
\]

Conversely, given $w_1, \ldots, w_n$, the corresponding primitive $I_\infty^n $-torsion of $ \Phi^{j_0}$ is written as 
\[
    \ker \Omega_k = \ker ((\delta_1\delta_2\cdots \delta_{k-1})^{q-1}\tau - w_{k})\cdots   (  (\delta_2\delta_1)^{q-1} \tau - w_3 )  (\delta_1^{q-1} \tau - w_2 )(\tau - w_1),
\]
where 
\[
    \delta_k^{q-1} =   \frac{1}{ 1 + (1-\zeta^{q^{k-1}-q^{k}})w_k^q }.
\]
This establishes the one-to-one correspondence between the parameters $w_i$ and the set of primitive $I_\infty^n $-torsions of $ \Phi^{j_0}$. 
\end{proof}

\section{Reduction of modular curves}\label{sec:thmC}
\subsection{Genus formula}
In this section, we briefly recall the genus formula established by Bassa-Beelen-Nguyen \cite{MR3433893} based on the results of \cite{Gekeler1986}.
For this aim we temporarily reset $ \A $ to be a Dedekind domain arising from a general smooth curve over $ \mathbb{F}_q $
associated with a unique infinity. Let $ \delta  $ be the degree of the infinity. 
\begin{notation}\label{not:functions}
    Let $\n \subseteq \A$ be an ideal and suppose that $ \n= \p^{r_1} \cdots \p^{r_s} $ for prime ideals
$\p_1, \ldots, \p_s$ and positive integers $r_1,\ldots,r_s$. Writing $q_i := |\p_i|$, we define
\[
\epsilon(\n) = \prod_{i=1}^s q_i^{r_i-1} (q_i+1) 
\]
and 
\[
\kappa(\n) =\prod_{i=1}^s \left( q_i^{\floor{r_i/2}} + q_i ^{r_i - \floor{r_i/2} -1 }\right),  
\]
where $\floor{r}$ denotes the integral part of a real number $r$.

\end{notation}
\begin{thm}[Theorem 3.1 of \cite{MR3433893}]\label{thm:genus}
  Let $\A, \n  $ be as in Notation \ref{not:functions}. Let $\x (\n)$ be the minimal modular curve associated with $\A$. Denote by $ P_K (t) $ the $L$-polynomial of the quotient field $K$ of $\A$. Then the genus of $ \x (\n) $ is given by 
  \[
    g(\x (\n) ) = 1 + \frac{(q^\delta-1 )\epsilon(\n) P_K(q)}{(q^2-1)(q-1)} - \frac{P_K(1) \delta }{q-1} \cdot (\kappa(\n) + 2^{s-1}(q-2) ) + \Delta,
  \]
  where $ \Delta = - P_K(-1) 2^{s-1}q/(q + 1) $ if $\delta$ is odd and all prime divisors of $\n$ are of even degree; and
$\Delta = 0$ otherwise.
\end{thm}
We now return to our main setting on $\A$ and choose $\n = I_{\infty}^k$. Then the degree of infinity equals $ \delta = \deg P_{\rho} = 2 $. It is obvious that $ P_K = 1 $ since $K$ is rational. 
From Notation \ref{not:functions}, we have $s =1 $, $|{I_\infty} |= q $, 
\[
\epsilon(\n) = q^{k-1} (q+1) 
\]
and 
\[
\kappa(\n) = q^{\floor{k/2}} + q ^{k - \floor{k/2} -1 }. 
\]
Substituting these quantities into Theorem \ref{thm:genus}, we have 
\begin{align}
    g(\x(\n)) =& 1 + \frac{ \epsilon(\n)  }{  q-1 } - \frac{ 2 }{q-1} \cdot (\kappa(\n) + q-2 ) \nonumber \\
    = & -1 + \frac{ q^{k-1} (q+1)  }{  q-1 } - \frac{ 2 }{q-1} \cdot ( q^{\floor{k/2}} + q ^{k - \floor{k/2} -1 }   -1  ) . \label{eq:genus} 
\end{align}
In particular, for $ k=1 $, we have 
$
g(x_0(\n)) = 0$.
This coincides with the fact that $ \x (I_{\infty}) $  is rational. 
\subsection{Supersingular points}
 
For $\eta \in \mathbb{F}_{q^2} \setminus (\{\zeta, \zeta^q \} \cup \mathbb{F}_q) $, we define 
\begin{align*}
 z_\eta  &= \frac{(t - \eta)(t - \eta^q)}{(t - \zeta)(t - \zeta^q)} \\
 & =\frac{ -(\eta+\eta^q - \zeta - \zeta^q) t + \eta^{q+1} - \zeta^{q+1} }{(t - \zeta)(t - \zeta^q)}  + 1 \\
&= (\eta^{q+1} - \zeta^{q+1}) x -(\eta+\eta^q - \zeta - \zeta^q) y +1. 
\end{align*}
It is clear  that $ z_ \eta $ is an element of $ \A $ of degree two. 
Denote by $I_\eta$ the ideal of $ \A $ generated by $ z_\eta$. 
The Drinfeld $\A$-module $ \Phi^{j} $ has a good reduction at $I_\eta$. Let us denote the $I_\eta$-reduction of $\Phi^{j}$ by $ \bar \Phi^{j} $.
More precisely, the $\mathbb{F}_q$-homomorphism 
\[ \A \to \mathbb{F}_{q^2}:  x \mapsto \frac{1}{(\eta- \zeta)(\eta -\zeta^q)}, y \mapsto \frac{\eta}{(\eta- \zeta)(\eta -\zeta^q)} 
\]
defines an $\A$-field with characteristic $ I_\eta $.
Substituting $ t  = \eta $ into both $\Phi^j_x $ and $\Phi^j_y $, we obtain the Drinfeld $\A$-module $ \bar{\Phi}^j $ over $\mathbb{F}_{q^2}$ with expressions 
\begin{align*}
    \bar{\Phi}_x^j =&\left( \frac{-j^{q(q+1)}}{(\eta^q-\zeta^q)(\eta-\zeta^q)}\tau^2+\left(\frac{j^q}{\zeta(\eta^q-\zeta^q)}+\frac{j^{q+1}}{(1-\zeta^{1-q})(\eta^q-\zeta)(\eta-\zeta)}\right)\tau +  \frac{j}{(\eta -\zeta)(\eta-\zeta^q)}\right)\\
    &\cdot\left(\tau^2+\left(\frac{1}{1-\zeta^{1-q}}+\frac{\eta-\zeta^q}{\zeta}\frac{1}{j}\right)\tau +\frac{1}{j}\right),
\end{align*}
and
\begin{align*}
    \bar{\Phi}_y^j =  & \left(\frac{-j^{q(q+1)}\zeta^q}{(\eta^q-\zeta^q)(\eta-\zeta^q)}\tau^2+\left(\frac{j^q\zeta^q}{\zeta(\eta^q-\zeta^q)}-\frac{\zeta^q j^{q+1}}{(1-\zeta^{q-1})(\eta^q-\zeta)(\eta-\zeta)}\right)\tau + \frac{j \eta}{(\eta -\zeta)(\eta-\zeta^q)}\right)\\
     &\cdot\left(\tau^2+\left(\frac{1}{1-\zeta^{1-q}}+\frac{\eta-\zeta^q}{\zeta}\frac{1}{j}\right)\tau +\frac{1}{j}\right).
\end{align*}
We determine the supersingular condition for $ \bar{\Phi}^j $ in the following lemma.
\begin{lem}\label{lem:supersingular}
    The $\zeta^q$-type Drinfeld $\A$-module $ \bar\Phi^j $ over $\mathbb{F}_{q^2}$ is supersingular, if and only if  $ j $ is contained in $ \mathbb{F}_{q^2}$ and satisfies
    \begin{equation}\label{eq:supersingular}
        \left( j +\frac{\eta -\zeta } {\zeta -\zeta^{1-q} }\right)^{q+1} + \frac{   \eta-\eta^q  }{ \zeta-\zeta^q   }=0 . 
    \end{equation}
\end{lem}

\begin{proof}
    It is straightforward to check that 
    \begin{align*}
   \bar \Phi_{z_\eta}^j =  & (\eta^{q+1}-\zeta^{q+1}) \bar\Phi^j_{x}  -(\eta+\eta^q-\zeta-\zeta^q)\bar\Phi^j_y+1 \\
    =&-j^{q(q+1)}\tau^4+\frac{j^q}{1-\zeta^{1-q}}\left(-j^{q^2} +j\right)\tau^3\\
    +&\left( \left(\frac{j}{1-\zeta^{1-q}}+\eta\zeta^{-1}-1\right)^{q+1}+(\zeta^{-1}-\zeta^{-q})(\eta-\eta^q)\right)\tau^2. 
\end{align*}
This implies that  $ \bar \Phi ^j $ is supersingular if and only if 
\[
  \left(\frac{j}{1-\zeta^{1-q}}+\eta\zeta^{-1}-1\right)^{q+1}+(\zeta^{-1}-\zeta^{-q})(\eta-\eta^q) =0 . 
\]
The last equality can be simplified to \eqref{eq:supersingular}. Moreover, all such $ j $ are contained in $ \mathbb{F}_{q^2}$. 
\end{proof}

In a similar manner, we define $\x (I_\infty^k)/I_\eta$ to be the $I_\eta$-reduction of the modular curve $\x(I_\infty^k) $ for each $k\geqslant 0$. The following result is a consequence of Theorem \ref{thm:thmB}.
\begin{cor}
    The function field $ \bar{\G}_k $ of $\x(I_{\infty}^k)/I_\eta$ is given by $\mathbb{F}_{q^2}(j_0, w_1, \cdots, w_k )$, where $j_0, w_1, \cdots, w_k $ are subject to \begin{equation}\label{eq:w1reduction}
   w_1^{q+1} + \left(\frac{1}{1-\zeta^{1-q}}+\frac{\eta -\zeta^q }{\zeta  }\frac{1}{j_0}\right)w_1 +\frac{1}{j_0} =0 ,
    \end{equation}
   and 
   \begin{equation}\label{eq:wkreduction} 
    \sum_{i=0}^{q-1}  (w_{k-1}^\nabla)^{i}w_k^{q-i}= \frac{w_{k-1}^q}{1-(\zeta^{q^{k+1}-q^k}-1)w_{k-1}  } \left( w_{k-1}^\nabla (\eta - \zeta^{q^{k+1}} ) \right)^{q-1} .
   \end{equation}
    Notice that here $w_{k-1}^{\nabla}$ is given by 
    \[
        w_{k-1}^{\nabla} = \frac{1}{(\zeta^{q^{k}-q^{k+1}}-1)(1+\zeta^{-q^k} (\eta - \zeta^{q^{k+1}})w_{k-1})}.
    \]
\end{cor}
In particular, the function field $ \bar{\G}_0 $ of $\x(1)/I_\eta$ is $ \mathbb{F}_{q^2}[j] $. 
 We already known that the supersingular points of $\x(1)/I_\eta$ are precisely $\mathbb{F}_{q^2}$-rational and satisfy the equality \eqref{eq:supersingular}
 by Lemma \ref{lem:supersingular}.
From Lemma \ref{lem:isogenyPhi}, we know 
$\delta_1 (\tau - w_1) $ is an isogeny from $ \bar\Phi^{j_0} $ to another $\zeta$-type Drinfeld module $ \bar\Phi^{\sigma; j_1} $,  where $ j_1 $ is the $ j $-invariant satisfies 
\begin{equation}\label{eq:j1}
    j_1 = j_0^q \frac{ \eta - \zeta}{\eta^q - \zeta}   \cdot \left( 1 +(1-\zeta^{q-1})   w_1 \right)  ^{q^2+1} . 
\end{equation}
It is clear that $ \bar\Phi^{\sigma; j_1} $ is also supersingular. Applying the same argument as in Lemma \ref{lem:supersingular} to $ \bar\Phi^{\sigma; j_1} $ yields that $ j_1 $ is also contained in $ \mathbb{F}_{q^2}$. 
It follows from \eqref{eq:j1} that 
\[
  \left( 1 +(1-\zeta^{q-1})   w_1 \right) ^{q^4-1} = \left( \frac{j_1}{j_0^q}\frac{\eta^q - \zeta}{ \eta - \zeta}\right)^{q^2-1} = 1. 
\]
This implies that $ (1 +(1-\zeta^{q-1}) w_1) $ lies in $ \mathbb{F}_{q^4} $, and so does $ w_1 $. Therefore, we understand that when $j_0$ verifies \eqref{eq:supersingular}, all the solutions $w_1$ of \eqref{eq:w1reduction} for  are contained in $\mathbb{F}_{q^4}$ and are distinct. 
Proceeding with the same procedure, we find that for given coordinates $ j_0, w_1,\cdots, w_{k-1} \in \mathbb{F}_{q^4} $, the solutions $w_k$ of \eqref{eq:wkreduction} form a subset of $\mathbb{F}_{q^4} $ with cardinality $q$. 
Hence following lemma follows.
\begin{lem}\label{lem:cardinality}
    The supersingular points of the modular curves $\x (I_\infty^k)/I_\eta$ are $ \mathbb{F}_{q^4} $-rational and their cardinality is  
    $
        (q+1)^2 q^{k-1}
    $. 
\end{lem}
\subsection{Proof of Theorem \ref{thm:thmC}}
The proof of Theorem \ref{thm:thmC} relies on Lemma \ref{lem:cardinality} and the equality given in \eqref{eq:genus}. We elaborate on the details below.
\begin{proof} 
    Let $ \mathbb{F}_{q^4}\bar{\G}_k $ denote the function field of $\x (I_\infty^k)/I_\eta$ over the finite field $\mathbb{F}_{q^4} $. We know the genus of a curve is invariant under taking reduction and constant field extension. So the genus $ g(\mathbb{F}_{q^4}\bar{\G}_k) $ of  $ \mathbb{F}_{q^4}\bar{\G}_k $ is given by \eqref{eq:genus}. From Lemma \ref{lem:cardinality}, the cardinality $ N(\mathbb{F}_{q^4}\bar{\G}_k) $ of the $ \mathbb{F}_{q^4} $-rational points on $\x (I_\infty^k)/I_\eta$ is greater than $
        (q+1)^2 q^{k-1}
        $. 
        
        We obtain that  
        \begin{equation}\label{eq:ihara}
           \lambda(\mathbb{F}_{q^4}\bar{\G}_k) = \lim_{k\to \infty} \frac{ N(\mathbb{F}_{q^4}\bar{\G}_k)  }{g(\mathbb{F}_{q^4}\bar{\G}_k) } \geqslant q^2 -1 .
        \end{equation}
        Since the upper bound of Ihara's quantity over $\mathbb{F}_{q^4} $ is $ (q^2 -1 )$, we know that the equality of \eqref{eq:ihara} holds. In other words, the function field tower $ \{ \mathbb{F}_{q^4}\bar{\G}_k \} $ is optimal. 
\end{proof}


 \bibliographystyle{amsplain}
 \bibliography{drinfeld}

\end{document}